\documentclass[11pt,letterpaper]{article}

\makeatletter
\newif\ifanonymous
\@ifclasswith{article}{anonymous}{\anonymoustrue}{\anonymousfalse}
\makeatother

\usepackage[letterpaper,margin=1.0in]{geometry}
\usepackage{color,latexsym,amsmath,amssymb}
\usepackage{fancyhdr}
\usepackage{todonotes}
\usepackage{amsthm}
\usepackage{aliascnt}
\usepackage{svg}
\usepackage{changepage}
\usepackage{graphicx}
\usepackage{array,booktabs,float,pdflscape}
\usepackage[utf8]{inputenc}
\usepackage[colorlinks]{hyperref}
\usepackage[nameinlink,capitalize,noabbrev]{cleveref}
\usepackage{listings}
\usepackage{xcolor}
\usepackage{tikz}
\usetikzlibrary{arrows.meta}
\usepackage[thinlines]{easytable}
\usepackage{makecell}
\usepackage{orcidlink}
\usepackage{authblk}
\usepackage{fontawesome5} %
\usepackage{algorithm}
\usepackage{algpseudocode}
\usepackage{tikz}
\usetikzlibrary{arrows.meta,decorations.pathreplacing,positioning}

\definecolor{codegreen}{rgb}{0,0.6,0}
\definecolor{codegray}{rgb}{0.5,0.5,0.5}
\definecolor{codepurple}{rgb}{0.58,0,0.82}
\definecolor{backcolour}{rgb}{0.95,0.95,0.95}
\definecolor{figureink}{HTML}{17324D}
\definecolor{routeorange}{HTML}{E58A3A}
\definecolor{metricteal}{HTML}{168C8C}
\definecolor{localpurple}{HTML}{8064A2}
\definecolor{resultblue}{HTML}{3B78A8}
\definecolor{figurepanel}{HTML}{E7EBEF}
\newcolumntype{V}{!{\color{black!35}\vrule width .35pt}}
\lstdefinestyle{overleafstyle}{
    backgroundcolor=\color{backcolour},
    commentstyle=\color{codegreen},
    keywordstyle=\color{magenta},
    numberstyle=\tiny\color{codegray},
    stringstyle=\color{codepurple},
    basicstyle=\ttfamily\footnotesize,
    breakatwhitespace=false,
    breaklines=true,
    captionpos=b,
    keepspaces=true,
    numbers=left,
    numbersep=5pt,
    showspaces=false,
    showstringspaces=false,
    showtabs=false,
    tabsize=2
}

\newcommand{\RR}{\mathbb{R}}

\newcommand{\defn}[1]{\textbf{#1}}
\newcommand{\con}{\text{con}}

\newcommand{\remove}[1]{}

\newtheorem{theorem}{Theorem}[section]

\newtheorem{corollary}[theorem]{Corollary}
\newtheorem{conjecture}[theorem]{Conjecture}

\newtheorem{definition}[theorem]{Definition}
\newtheorem{lemma}[theorem]{Lemma}
\newtheorem{observation}[theorem]{Observation}
\newaliascnt{proposition}{theorem}
\newtheorem{proposition}[proposition]{Proposition}
\aliascntresetthe{proposition}

\title{On Eigenvalue Bounds for Bounded Genus Graphs\\and Minor-Free Graphs}
\ifanonymous
    \author{Anonymous author(s)}
\else
    \author{Benedikt Kolbe\footnote{Hausdorff Center for Mathematics, University of Bonn. This work was partially supported by the Lamarr Institute for Machine Learning and Artificial Intelligence.} \, \orcidlink{0009-0005-0440-4912}\, \href{mailto:bkolbe@uni-bonn.de}{\faEnvelope}, and Jack Spalding-Jamieson\,\orcidlink{0000-0002-1209-4345}\, \href{mailto:jacksj@uwaterloo.ca}{\faEnvelope}}
\fi

\date{}

\begin{document}

\maketitle

\begin{abstract}
In this paper, we resolve a 30-year-old conjecture of Spielman and Teng
concerning the performance
of the spectral partitioning method
on graphs embeddable on an orientable surface of genus $g\ge 1$.
In particular, for such a graph $G$ with $n$ vertices and maximum degree $\Delta$,
we show that the second-smallest eigenvalue of its Laplacian matrix
satisfies $\lambda_2(L_G)\lesssim\Delta\frac g n$.
We also obtain an improved eigenvalue bound for $K_h$-minor-free graphs of $\lambda_2(L_G)\lesssim\Delta\frac{h^2(\log h)^2}n$.

In fact, our results directly prove much stronger bounds
for reweighted eigenvalues, including higher reweighted eigenvalues.
As a consequence, we obtain bounds not just on Laplacian eigenvalues,
but also on normalized Laplacian eigenvalues and Steklov eigenvalues.
Our results for genus-$g$ graphs are optimal for all of these kinds of eigenvalues,
while our results for $K_h$-minor-free graphs are optimal up to $\log(h)$ factors.

Our techniques for genus-$g$ graphs bootstrap bounded-degree bounds of normalized eigenvalues for entire classes to bounds for reweighted eigenvalues for the same classes without the bounded-degree limitation,
while our techniques for $K_h$-minor-free graphs generalize an argument of Korhonen and Lokshtanov,
making use of the Lov\'asz local lemma.

\end{abstract}

\section{Introduction}

Spectral partitioning is a highly-practical and effective method of graph partitioning
that has seen widespread adoption~\cite{spielman2007spectral}.
One of the most interesting theoretical questions concerns its performance:
Why is it so effective?
To address this question, we can certify its performance for special graph classes.
This requires upper bounds on eigenvalues of the ``Laplacian matrix''.

\begin{definition}
Let $G=(V,E)$ be a graph with adjacency matrix $A_G$ and diagonal degree matrix $D_G$.
The \defn{Laplacian matrix} of $G$ is
$L_G=D_G-A_G$.
We write its $k$th-smallest eigenvalue as $\lambda_k(L_G)$
for $1\leq k\leq|V|$.
\end{definition}

The quantity of interest for classical spectral partitioning is $\lambda_2(L_G)$.
This quantity directly measures the success of spectral partitioning:
Smaller values imply better partitions.
Hence, to certify good performance in a special graph class, we require upper bounds on $\lambda_2(L_G)$.

Throughout this paper, we assume graphs are simple and connected, and use the notation $\Delta(G)$ for the maximum degree of $G$.
In 1996, Spielman and Teng~\cite{spielman1996disk,spielman2007spectral}
showed that for a planar graph $G$ with $n$ vertices,
$\lambda_2(L_G)\leq\frac{8\Delta(G)}n$,
in addition to a number of bounds for geometric classes.
Importantly, they also posed the following pair of conjectures
(letting $a\lesssim b$ denote $a\in O(b)$).

\begin{conjecture}[{\cite[Conjecture~1]{spielman2007spectral}}]
\label{conj:genus}
Let $G$ be an $n$-vertex graph of orientable genus at most $g$
and maximum degree $\Delta$,
where $g\geq1$.
Then
\[
\lambda_2(L_G)\lesssim\frac{\Delta g}{n}.
\]
\end{conjecture}

\begin{conjecture}[{\cite[Conjecture~2]{spielman2007spectral}}]
\label{conj:minor}
There is an absolute constant $c>0$ such that every $n$-vertex
$K_h$-minor-free graph $G$ of maximum degree $\Delta$ satisfies
\[
\lambda_2(L_G)\lesssim\frac{\Delta h^c}{n}.
\]
\end{conjecture}

We note that
\cref{conj:genus}
would be the best-possible bound,
and
\cref{conj:minor}
would be the best-possible bound if $c=2$
(see \cref{sec:lower-bounds}).

A long series of follow-up works
has made progress towards addressing these conjectures, as well as a number of natural generalizations.
First, \cref{conj:genus} was proven for the special case of bounded-degree triangulated genus-$g$ graphs
by Kelner~\cite{%
kelner2004spectral%
,kelner2006spectral%
,kelner2006new%
}%
\footnote{Kelner claimed that his result held for non-triangulated graphs as well, but an issue was recently discovered in this step of the proof~\cite{spalding2025reweighted}.}.
\Cref{conj:minor} was later proven in its stated generality by Biswal, Lee, and Rao~\cite{biswal2008eigenvalue,biswal2010eigenvalue}
with a bound of $O\left(\frac{\Delta h^6\log h}n\right)$.
They showed a bound of $O\left(\frac{\Delta g^3}n\right)$
for genus-$g$ graphs.
A number of further improvements have followed for all three of these results%
~\cite{lee2010genus,amini2018transfer,spalding2025reweighted},
as well as bounds on higher eigenvalues and other kinds of eigenvalues%
~\cite{kelner2009higher,kelner2011metric,amini2018transfer,spalding2025reweighted,tung2025reweighted,%
lin2024upper,lin2026upper,chen2026first,zhan2026higherstekloveigenvaluesgraphs}.
However, these works left two large remaining open questions:
\begin{enumerate}
    \item Can \cref{conj:genus} be shown in full generality?
    \item Can the constant $c$ in \cref{conj:minor} be improved to its best-possible value of two?
\end{enumerate}

In our work, we provide an affirmative answer to the first question,
and make significant progress towards the second question
by proving a bound that is optimal for minor-free graphs up to a factor of $O(\log^2 h)$.
However, our results also go far beyond bounds on $\lambda_2(L_G)$.
Rather, our bounds on $\lambda_2(L_G)$ are simple corollaries of our main results,
which concern a general form of \emph{reweighted} eigenvalues.
Reweighted eigenvalues have been shown through a long line of work to give an analogous
(and often more powerful) alternative algorithm to spectral partitioning~\cite{roch2005bounding,olesker2022geometric,olesker2024geometric,kwok2022cheeger,kwok2025cheeger,jain2022dimension,spalding2025reweighted}.
The general form of reweighted eigenvalues we present here is specifically due to Kwok, Lau, and Tung~\cite{kwok2025cheeger}.

\begin{definition}
A \defn{positive probability distribution} on a finite set $X$ is a function
$\pi:X\to\RR$ satisfying $\pi(x)>0$ for every $x\in X$ and
$\sum_{x\in X}\pi(x)=1$.
\end{definition}

\begin{definition}[\cite{kwok2022cheeger,kwok2025cheeger}]
\label{def:reweighted-eigenvalues}
Let $G=(V,E)$ be a graph, and let $\pi$ be a positive probability distribution on $V$.
For $1\leq k\leq|V|$, the \defn{$k$th reweighted eigenvalue} is
\[
\begin{array}{rccr@{\:}c@{\:}ll}
\lambda_k^*(G,\pi)
&=&
\displaystyle\max_{P\in\RR_{\geq0}^{V\times V}}
&\multicolumn{4}{l}{\lambda_k(I-P)}\\
&&\text{subject to}
&\displaystyle\sum_{v\in V}P_{uv}
&{}={}&1
&\forall u\in V,\\
&&
&\pi(u)P_{uv}
&{}={}&\pi(v)P_{vu}
&\forall u,v\in V,\\
&&
&P_{uv}
&{}={}&0
&\forall uv\notin E,\ u\neq v.
\end{array}
\]
For uniform $\pi$, write $\lambda_k^*(G)$.
\end{definition}

Substituting $w_{uv}=\pi(u)P_{uv}$ gives the following equivalent form.

\begin{observation}
\label{obs:reweighted-weighted-laplacian}
Let $G=(V,E)$ be a graph,
let $\pi$ be a positive probability distribution on $V$,
and let $\Pi$ be the diagonal matrix with $\Pi_{vv}=\pi(v)$.
For $w\in\RR_{\geq0}^E$,
let $L_w$ be the weighted Laplacian with edge weights $w$.
Then
\[
\begin{array}{rccr@{\:}c@{\:}ll}
\lambda_k^*(G,\pi)
&=&
\displaystyle\max_{w\in\RR_{\geq0}^E}
&\multicolumn{4}{l}{\lambda_k\left(\Pi^{-1/2}L_w\Pi^{-1/2}\right)}\\
&&\text{subject to}
&\displaystyle\sum_{e\ni v}w_e
&{}\leq{}&\pi(v)
&\forall v\in V.
\end{array}
\]
\end{observation}

Kwok, Lau, and Tung showed
that these quantities relate to very broad classes of weighted graph partitioning problems,
including not just bisections, but also multiway partitioning~\cite{kwok2025cheeger}.
Just like the Laplacian eigenvalues,
bounds for special graph classes also imply interesting results related to graph partitioning, particularly involving vertex separators~\cite{tung2025reweighted,spalding2025reweighted}.
Consequently, improvements for the best-known bounds on these eigenvalues
imply improved guarantees to the performance of existing algorithms.

Our main results are the following bounds:

\begin{theorem}
\label{thm:main-genus}
Let $G=(V,E)$ be an $n$-vertex graph of orientable genus at most $g$.
Every positive probability distribution $\pi$ on $V$ satisfies
$\lambda_k^*(G,\pi)\lesssim(g+k)\|\pi\|_\infty$
for $2\leq k\leq n$.
\end{theorem}

\begin{theorem}
\label{thm:main-minor}
Let $G=(V,E)$ be an $n$-vertex $K_h$-minor-free graph for $h\geq3$.
Every positive probability distribution $\pi$ on $V$ satisfies
$\lambda_k^*(G,\pi)\lesssim h^2k\log^2h\,\|\pi\|_\infty$
for $2\leq k\leq n$.
\end{theorem}

The bounded genus bound of
\cref{thm:main-genus} is optimal up to constant factors,
while the minor-free bound of
\cref{thm:main-minor}
is optimal up to a factor of
$O(\log^2h)$
(see \cref{sec:lower-bounds}).
A full history of eigenvalue bounds for special graph classes, and how our results compare, can be found in
\cref{tab:history}.

We also note an immediate algorithmic consequence of our results, due to Spalding-Jamieson~\cite{spalding2025reweighted}.
A \defn{$\frac23$-balanced vertex separator} of a graph $G=(V,E)$ with $n$ vertices
is a subset of vertices $S\subset V$ so that $G-S$ has no connected component of size greater than $\frac23n$.
The combination of
\cref{thm:main-genus,thm:main-minor} with \cite[Theorem 1.4]{spalding2025reweighted}
and \cite[Theorem 2]{conroy2025protect} gives the following.

\begin{corollary}
Given a graph $G=(V,E)$ with $n$ vertices and maximum degree $\Delta$,
there is a polynomial-time algorithm that
computes a balanced vertex separator $S$ of $G$ such that:
\begin{itemize}
    \item If $G$ has genus at most $g\ge 2$, then $|S|\lesssim\min\{\sqrt{\log\Delta},\log g\}\cdot\sqrt{gn}$.
    \item If $G$ is $K_h$-minor-free, then $|S|\lesssim\min\{\sqrt{\log\Delta},\log h\}\cdot h\log h\cdot\sqrt n$.
\end{itemize}
\end{corollary}

Note that the algorithm is truly polynomial-time: It does not have any fixed-parameter time-complexity dependence
on $g$ or $h$, and it does not take a surface embedding (or any similar structure) as input.
In fact, it is the exact same algorithm for both classes of graphs.

\subsection{Other Eigenvalues and Relationships}
\label{subsec:other-eigen}

There are a few other relevant well-studied notions of eigenvalues for graphs,
including the original Laplacian eigenvalue.
Notably, our bounds on reweighted eigenvalues
imply bounds for all of these other notions,
all of which are optimal for genus-$g$ graphs.

\begin{definition}
Let $G=(V,E)$ be a graph with no isolated vertices.
The \defn{normalized Laplacian} of $G$ is
$\mathcal L_G:=D_G^{-1/2}L_GD_G^{-1/2}$.
We write its $k$th eigenvalue as $\lambda_k(\mathcal L_G)$
for $1\leq k\leq|V|$.
\end{definition}

\begin{definition}[\cite{zhan2026higherstekloveigenvaluesgraphs}]
\label{def:steklov-eigenvalues}
Let $G=(V,E)$ be a connected graph,
and let $B\subset V$ be nonempty.
A real number $\sigma$ is a \defn{Steklov eigenvalue} of $(G,B)$ if there is a
nonzero function $f:V\to\RR$ such that
\[
\sum_{\substack{v\in V\\uv\in E}}(f(u)-f(v))
=
\begin{cases}
0, & u\in V\setminus B,\\
\sigma f(u), & u\in B.
\end{cases}
\]
There are $|B|$ Steklov eigenvalues (with multiplicity), which we write as
\[
0=\sigma_1(G,B)\leq\cdots\leq\sigma_{|B|}(G,B).
\]
\end{definition}

Our main results transfer as follows.

\begin{theorem}
\label{thm:other-eigenvalue-transfer}
Fix $k\geq2$ and $A>0$.
Let $\mathcal G$ be either the family of graphs of orientable genus at most $g$
or the family of $K_h$-minor-free graphs for some $h\geq3$.
Suppose every $G\in\mathcal G$ satisfies
\[
\lambda_k^*(G)
\lesssim
\frac{A}{|V(G)|}.
\]
Then every $n$-vertex graph $G\in\mathcal G$ with maximum degree $\Delta$ satisfies
\[
\lambda_k(\mathcal L_G)
\leq
\lambda_k(L_G)
\lesssim
\frac{A\Delta}{n}.
\]
Moreover,
for a graph $G\in\mathcal G$ with maximum degree $\Delta$, every $B\subset V(G)$ with $|B|\geq k$ satisfies
\[
\sigma_k(G,B)
\lesssim
\frac{A\Delta(G)}{|B|}.
\]
\end{theorem}

We get the following bounds as a result.

\begin{corollary}
\label{cor:genus-other-eigen-bounds}
Let $G$ be an $n$-vertex graph of orientable genus at most $g$ with maximum
degree $\Delta$.
For every $1\leq k\leq n$,
\[
\lambda_k(\mathcal L_G)
\leq
\lambda_k(L_G)
\lesssim
\frac{\Delta(g+k)}{n}.
\]
For every $B\subset V(G)$ with $|B|\geq2$ and every $2\leq k\leq|B|$,
\[
\sigma_k(G,B)
\lesssim
\frac{\Delta(g+k)}{|B|}.
\]
\end{corollary}

\begin{corollary}
Let $G$ be an $n$-vertex $K_h$-minor-free graph for $h\geq3$ with maximum
degree $\Delta$.
Then
\[
\begin{aligned}
\lambda_2(\mathcal L_G)
&\leq
\lambda_2(L_G)
\lesssim
\frac{\Delta h^2\log^2h}{n},\\
\lambda_k(\mathcal L_G)
&\leq
\lambda_k(L_G)
\lesssim
\frac{\Delta h^2k\log^2h}{n}
\qquad (3\leq k\leq n).
\end{aligned}
\]
For every $B\subset V(G)$ with $|B|\geq2$,
\[
\begin{aligned}
\sigma_2(G,B)
&\lesssim
\frac{\Delta h^2\log^2h}{|B|},\\
\sigma_k(G,B)
&\lesssim
\frac{\Delta h^2k\log^2h}{|B|}
\qquad (3\leq k\leq|B|).
\end{aligned}
\]
\end{corollary}

Notably,
\cref{cor:genus-other-eigen-bounds}
is precisely what resolves
\cref{conj:genus} in the affirmative with $k=2$.

\subsection{Formalization}

Our main results
\cref{thm:main-genus}
and
\cref{thm:main-minor}
have been formalized in Lean.%
\ifanonymous\else%
\footnote{\url{https://github.com/jacketsj/eigen-formal}}%
\fi

The formalization was generated using Codex with GPT 5.6-sol ``extra high'',
and its statements and definitions were carefully checked over by a human.
The definition of genus-$g$ graphs is based on the Euler characteristic and a rotation system.
Since the formalization was done autonomously by an agent,
the formalized proof differs from the paper.

\subsection{Technical Overview}

Our techniques for bounded genus graphs and minor-free graphs differ significantly,
while our methods for extending our bounds to other types of eigenvalues are unified.

\paragraph{Bounded-Genus Graphs}
Bounded-genus graphs of bounded degree are already known to admit
eigenvalue bounds matching \cref{conj:genus} (and similar for higher eigenvalues)~\cite{amini2018transfer}.
Using a sequence of reductions, we directly leverage this bound in order to bootstrap all the way to
a bound on reweighted eigenvalues for bounded-genus graphs of unbounded degree.
Importantly, each of our bootstrapping steps is \emph{class-wide} for bounded-genus graphs,
rather than per-graph.

\paragraph{Minor-Free Graphs}
At a high level,
our techniques for minor-free graphs are based on the framework of
Biswal, Lee, and Rao~\cite{biswal2008eigenvalue,biswal2010eigenvalue},
and its extensions by
Kelner, Lee, Price, and Teng~\cite{kelner2009higher,kelner2011metric},
Tung~\cite{tung2025reweighted},
and
Spalding-Jamieson~\cite{spalding2025reweighted}.

This infrastructure is well-established,
and, in the case of $\lambda_2^*$,
relies on bounds for the $L_2$-congestion of a certain kind of multicommodity flow.
Essentially all previous bounds on these $L_2$-congestion quantities have long been thought to be suboptimal,
and likely related to Hadwiger's conjecture.
Our main contribution provides an optimal bound for this congestion problem
by extending an argument of Korhonen and Lokshtanov~\cite{korhonen2024induced}
that uses the Lovász Local Lemma.
Notably, unlike previous methods that seemingly related the problem directly to Hadwiger's conjecture,
and were inherently combinatorial,
the proof we give is inherently probabilistic.
In order to apply our results to all $\lambda_k^*$, rather than just $\lambda_2^*$,
we extend this bound to a more general kind of congestion problem
(initially studied by Kelner, Lee, Price, and Teng~\cite{kelner2011metric})
by using a form of expander decompositions.

\paragraph{Other Eigenvalues}
In order to obtain our bounds on other kinds of eigenvalues,
we establish that our bounds on reweighted eigenvalues
are stronger.
For Laplacian eigenvalues and normalized Laplacian eigenvalues,
this simply follows from the definitions.
However, for Steklov eigenvalues,
our proof is once again \emph{class-wide},
and in fact applies to any class closed under adding leaves.
At a high level, the main idea for this proof is to add a large number of leaves
to every boundary vertex, thereby weighing the boundaries quite heavily in any Laplacian quadratic form,
and to relate the modified graph to the original via a Schur complement, although the details end up being rather tedious.

\clearpage
\begin{landscape}
\thispagestyle{empty}
\begin{table}[H]
\centering
\hfuzz=2pt
\fontsize{7.5}{9.3}\selectfont
\setlength{\tabcolsep}{1pt}
\renewcommand{\arraystretch}{1.4}
\begin{tabular}{
  @{}
  >{\centering\arraybackslash}m{.032\linewidth}
  >{\centering\arraybackslash}m{.058\linewidth}
  V
  >{\centering\arraybackslash}m{.082\linewidth}
  >{\centering\arraybackslash}m{.073\linewidth}
  V
  *{2}{>{\centering\arraybackslash}m{.061\linewidth}}
  V
  >{\centering\arraybackslash}m{.083\linewidth}
  >{\centering\arraybackslash}m{.074\linewidth}
  >{\centering\arraybackslash}m{.096\linewidth}
  >{\centering\arraybackslash}m{.102\linewidth}
  V
  *{2}{>{\centering\arraybackslash}m{.075\linewidth}}
  V
  >{\centering\arraybackslash}m{.046\linewidth}
  @{}
}
\toprule
\textbf{Year}
& \textbf{Ref.}
& \multicolumn{2}{V c V}{\textbf{Laplacian $L_G$}}
& \multicolumn{2}{c V}{\textbf{Normalized $\mathcal L_G$}}
& \multicolumn{4}{c V}{\textbf{Reweighted}}
& \multicolumn{2}{c V}{\textbf{Steklov}}
& \textbf{Triang.}\\
&
& $\boldsymbol{\lambda_2(L_G)\lesssim}$
& $\boldsymbol{\lambda_k(L_G)\lesssim}$
& $\boldsymbol{\lambda_2(\mathcal L_G)\lesssim}$
& $\boldsymbol{\lambda_k(\mathcal L_G)\lesssim}$
& $\boldsymbol{\lambda_2^*(G)\lesssim}$
& $\boldsymbol{\lambda_k^*(G)\lesssim}$
& $\boldsymbol{\lambda_2^*(G,\pi)\lesssim}$
& $\boldsymbol{\lambda_k^*(G,\pi)\lesssim}$
& $\boldsymbol{\sigma_2(G,B)\lesssim}$
& $\boldsymbol{\sigma_k(G,B)\lesssim}$
&\\
\midrule
\multicolumn{13}{c}{\textbf{Genus-$g$ graphs}}\\
\specialrule{0.25pt}{0.7pt}{0.7pt}
2004 & \cite{kelner2006spectral}
& $\frac{\Delta^{O(1)}g}{n}$
& & & & & & & & & & yes\\
\specialrule{0.25pt}{0.7pt}{0.7pt}
2008 & \cite{biswal2010eigenvalue}
& $\frac{\Delta g^3}{n}$
& & & & & & & & & &\\
\specialrule{0.25pt}{0.7pt}{0.7pt}
2009 & \cite{kelner2009higher}
& $\frac{\Delta g^3}{n}$
& $\frac{\Delta g^3k}{n}$
& & & & & & & & &\\
\specialrule{0.25pt}{0.7pt}{0.7pt}
2010 & \cite{lee2010genus}
& $\frac{\Delta g\log^2 g}{n}$
& $\frac{\Delta gk\log^2 g}{n}$
& & & & & & & & &\\
\specialrule{0.25pt}{0.7pt}{0.7pt}
2011 & \cite{kelner2011metric}
& $\frac{\Delta g\log^2 g}{n}$
& $\frac{\Delta gk\log^2 g}{n}$
& & & & & & & & &\\
\specialrule{0.25pt}{0.7pt}{0.7pt}
2014 & \cite{amini2018transfer}
& $\frac{\Delta^2g}{n}$
& $\frac{\Delta^2(g+k)}{n}$
& $\frac{\Delta g}{n}$
& $\frac{\Delta(g+k)}{n}$
& & & & & & &\\
\specialrule{0.25pt}{0.7pt}{0.7pt}
2024 & \cite{lin2024upper}\footnotemark
& & & & & & & &
& $\frac{\Delta g^3}{|B|}$ & &\\
\specialrule{0.25pt}{0.7pt}{0.7pt}
2025 & \cite{tung2025reweighted}
& & & &
& $\frac{g\log^2 g}{n}$
& $\frac{gk\log^2 g}{n}$
& $g\log^2 g\|\pi\|_\infty$
& $gk\log^2 g\|\pi\|_\infty$
& & &\\
\specialrule{0.25pt}{0.7pt}{0.7pt}
2025 & \cite{spalding2025reweighted}
& $\frac{\Delta g\log\Delta}{n}$ & & &
& $\frac{g\log\Delta}{n}$
& & & & & &\\
\specialrule{0.25pt}{0.7pt}{0.7pt}
2025 & \cite{spalding2025reweighted}
& $\frac{\Delta g\log^2 g}{n}$ & & &
& $\frac{g\log^2 g}{n}$
& & & & & &\\
\specialrule{0.25pt}{0.7pt}{0.7pt}
2025 & \cite{chen2026first}
& & & & & & & &
& $\frac{\Delta^{O(1)}g}{|B|}$ & & yes\\
\specialrule{0.25pt}{0.7pt}{0.7pt}
2026 & \cite{zhan2026higherstekloveigenvaluesgraphs}
& & & & & & & &
& $\frac{\Delta g\log^2 g}{|B|}$
& $\frac{\Delta gk\log^2 g}{|B|}$ &\\
\specialrule{0.7pt}{1.4pt}{1.4pt}
2026 & \textbf{This paper}
& $\boldsymbol{\frac{\Delta g}{n}}$
& $\boldsymbol{\frac{\Delta(g+k)}{n}}$
& $\boldsymbol{\frac{\Delta g}{n}}$
& $\boldsymbol{\frac{\Delta(g+k)}{n}}$
& $\boldsymbol{\frac gn}$
& $\boldsymbol{\frac{g+k}{n}}$
& $\boldsymbol{g\|\pi\|_\infty}$
& $\boldsymbol{(g+k)\|\pi\|_\infty}$
& $\boldsymbol{\frac{\Delta g}{|B|}}$
& $\boldsymbol{\frac{\Delta(g+k)}{|B|}}$
&\\
\specialrule{1.2pt}{2.2pt}{2.2pt}
\multicolumn{13}{c}{\textbf{$K_h$-minor-free graphs}}\\
\specialrule{0.25pt}{0.7pt}{0.7pt}
2008 & \cite{biswal2010eigenvalue}
& $\frac{\Delta h^6\log h}{n}$
& & & & & & & & & &\\
\specialrule{0.25pt}{0.7pt}{0.7pt}
2009 & \cite{kelner2011metric}
& $\frac{\Delta h^6\log h}{n}$
& $\frac{\Delta h^6k\log h}{n}$
& & & & & & & & &\\
\specialrule{0.25pt}{0.7pt}{0.7pt}
2025 & \cite{tung2025reweighted}
& & & &
& $\frac{h^6\log h}{n}$
& $\frac{h^6k\log h}{n}$
& $h^6\log h\|\pi\|_\infty$
& $h^6k\log h\|\pi\|_\infty$
& & &\\
\specialrule{0.25pt}{0.7pt}{0.7pt}
2025 & \cite{tung2025reweighted}\newline
+\cite{conroy2025protect}
&
&
&
&
& $\frac{h^2\log^3 h}{n}$
& $\frac{h^2k\log^3 h}{n}$
& $h^2\log^3 h\|\pi\|_\infty$
& $h^2k\log^3 h\|\pi\|_\infty$
&
&
&\\
\specialrule{0.25pt}{0.7pt}{0.7pt}
2025 & \cite{spalding2025reweighted}
& \scalebox{.84}{%
  $\frac{\Delta(h\log h\log\log h)^2}{n}$%
}
& & &
& \scalebox{.92}{%
  $\frac{(h\log h\log\log h)^2}{n}$%
}
& & & & & &\\
\specialrule{0.25pt}{0.7pt}{0.7pt}
2026 & \cite{zhan2026higherstekloveigenvaluesgraphs}
& & & & & & & &
& $\frac{\Delta h^6\log h}{|B|}$
& $\frac{\Delta h^6k\log h}{|B|}$ &\\
\specialrule{0.7pt}{1.4pt}{1.4pt}
2026 & \textbf{This paper}
& $\boldsymbol{\frac{\Delta h^2\log^2 h}{n}}$
& $\boldsymbol{\frac{\Delta h^2k\log^2 h}{n}}$
& $\boldsymbol{\frac{\Delta h^2\log^2 h}{n}}$
& \scalebox{.94}{%
  $\boldsymbol{\frac{\Delta h^2k\log^2 h}{n}}$%
}
& $\boldsymbol{\frac{h^2\log^2 h}{n}}$
& $\boldsymbol{\frac{h^2k\log^2 h}{n}}$
& $\boldsymbol{h^2\log^2 h\|\pi\|_\infty}$
& $\boldsymbol{h^2k\log^2 h\|\pi\|_\infty}$
& $\boldsymbol{\frac{\Delta h^2\log^2 h}{|B|}}$
& $\boldsymbol{\frac{\Delta h^2k\log^2 h}{|B|}}$
&\\
\bottomrule
\end{tabular}
\caption{
History of bounds for genus-$g$ and $K_h$-minor-free graphs.
``yes'' marks triangulation-only arguments whose claimed general
extensions have errors~\cite{spalding2025reweighted}.
}
\label{tab:history}
\end{table}
\footnotetext{This bound holds only in restricted cases.}
\end{landscape}

\section{Bounded-Genus Graphs}

In this section, we will prove
\cref{thm:main-genus}.

The proof essentially follows from the following proposition, the proof of which relies on a particular way of relating spectral information of the graph $G$ to that of another graph with smaller bounded degree.

\begin{proposition}[Degree reduction]\label{prop:degree-reduction}
Let $G$ be a simple connected graph embedded in an orientable surface $\Sigma$, let $\pi$ be a positive probability distribution on $V(G)$, and let $2\le k\le |V(G)|$.  Suppose that for the Laplacian eigenvalues we have 
\[
\lambda_k(L_H)\le \frac{A}{|V(H)|}
\]
for every simple connected graph $H$ embedded in $\Sigma$ and associated combinatorial Laplacian $L_H$, with $\Delta(H)\le3$ and $|V(H)|\ge k$.  Then the reweighted eigenvalues satisfy
\[
\lambda_k^*(G,\pi)\lesssim A\|\pi\|_\infty.
\]
\end{proposition}
\begin{proof}[Proof of Theorem~\ref{thm:main-genus}]
We only consider the case that the graph $G$ is simple and connected.
If $G$ is not simple, then multiple edges are replaced by one edge with edge weights of the constituent edges added, and
loops are simply removed. If $G$ is not connected, we can add edges to obtain a connected embedded graph on $\Sigma$. Note that merging parallel edges preserves $\lambda_k^*(G,\pi)$, while adding edges to connect the components can only increase it. The case $k=1$ is immediate, since the left hand side is $0$.
For $k\ge2$, embed $G$ in a surface of genus at most $g$.
Write $\mathcal L_H=D_H^{-1/2}L_HD_H^{-1/2}$ for the normalized Laplacian.
By~\cite[Theorem~1.1]{amini2018transfer}, we know that for bounded degree connected graphs $H$ of genus at most $g$, and  $1\le k\le \lvert V(H) \rvert$
\[
\lambda_k(\mathcal L_H)\lesssim \frac{\Delta(H)(g+k)}{\lvert V(H)\rvert }
\]
holds. By the Courant-Fischer theorem, $\lambda_k(L_H)\le \Delta(H)\lambda_k(\mathcal L_H)$.
Since $H$ was arbitrary, \Cref{prop:degree-reduction} applies,
and the result follows, with $A\lesssim g+k$.
\end{proof}

\subsection{Degree reduction}\label{sec:degree-reduction}

We will make use of a certain gadget in the proof, which is a tree with specific properties.

\begin{lemma}[Ordered binary tree]\label{lem:tree}
For every $s\ge1$ and every prescribed order of $s$ leaves, there is a plane binary tree $T$ where the leaves appear, written from left to right, in that order, $|V(T)|=2s-1$, the maximum degree is at most three, and such that every function on vertices $f:V(T)\to\RR$ satisfies
\begin{equation}\label{eq:tree-poincare}
\sum_{z\in V(T)}\bigl(f(z)-\bar f\bigr)^2
\le
4|V(T)|\sum_{ab\in E(T)}\bigl(f(a)-f(b)\bigr)^2,
\qquad
\bar f=\frac1{|V(T)|}\sum_{z\in V(T)}f(z).
\end{equation}
\end{lemma}

\begin{proof}
We start by recursively splitting each ordered interval of the $s$ leaves into two consecutive parts whose sizes differ by at most one. It is straightforward to check that the resulting full binary tree $T$ has $2s-1$ vertices.

Root the tree at its top vertex $r$.  A subtree $T_a$ rooted at a vertex $a$ at depth $d$ has at most $\lceil s/2^d\rceil$ leaves and hence at most $|V(T_a)|\le 2\lceil s/2^d\rceil-1$ vertices. Any vertex occurring at depth $d$ satisfies $2^d\le2s$, so we have 
\begin{equation}\label{eq:subtree-bound}
2^d|V(T_a)|\le4s\le4|V(T)|
\end{equation}
for every vertex $a$ at depth $d$.  For given $f$ and $z$, consider the \emph{edge differences} $\delta_1,\ldots,\delta_\ell$, defined as follows. First, take the unique simple edge path from the root $r$ to $z$, written as a sequence of adjacent vertices as $r=v_0,v_1,\dots,v_{\ell}=z$, then $\delta_j=f(v_j)-f(v_{j-1})$. Telescoping and weighted Cauchy-Schwarz gives
\[
\bigl(f(z)-f(r)\bigr)^2
\le
\left(\sum_{j=1}^{\ell}2^{-j}\right)
\left(\sum_{j=1}^{\ell}2^j\delta_j^2\right)
\le
\sum_{j=1}^{\ell}2^j\delta_j^2.
\]
After summing over $z$, the edge entering a depth-$d$ subtree appears in the sum with a coefficient of $2^d|V(T_a)|$, which is at most $4|V(T)|$, by \eqref{eq:subtree-bound}.  Thus
\[
\sum_z\bigl(f(z)-f(r)\bigr)^2
\le4|V(T)|\sum_{ab\in E(T)}\bigl(f(a)-f(b)\bigr)^2.
\]
The average $\bar f$ minimizes $t\mapsto\sum_z(f(z)-t)^2$, concluding the proof of the lemma.
\end{proof}
\begin{proof}[Proof of Proposition~\ref{prop:degree-reduction}]
Let $\pi$ be a positive probability distribution on the vertex set $V(G)$ of the graph $G$. An edge weighting $w\in\RR_{\ge 0}^{E(G)}$ is \defn{admissible} if
$\sum_{e\ni v}w_e\le\pi(v)$ for every $v$. For this proof, we will use the equivalent formulation of reweighted eigenvalues via the Rayleigh quotient  
\[
\lambda_k(w,\pi)
=
\min_{\substack{U\le\RR^{V(G)}\\ \dim U=k}}
\max_{0\ne x\in U}
\frac{\sum_{uv\in E(G)}w_{uv}(x_u-x_v)^2}
     {\sum_{v\in V(G)}\pi(v)x_v^2},
\qquad
\lambda_k^*(G,\pi)=\max_{w\text{ admissible}}\lambda_k(w,\pi).
\]

Fix an admissible weighting
$w_{uv}$ for $uv\in E(G)$. For every $x\in\RR^{V(G)}$, by Cauchy-Schwarz,
\begin{equation}\label{eq:universal-two}
\sum_{uv\in E(G)}w_{uv} (x_u-x_v)^2
\le2\sum_{v\in V(G)}\pi(v)x_v^2.
\end{equation}
The proposition is immediate whenever $A \|\pi\|_\infty \ge 1/16$, so we assume otherwise. Also, we first treat the case in which both $\pi$ and $w$ are positive and rational. The other case will follow by continuity. This assumption allows us to choose $N>1$ so that for every $v$ and $e$ 
\[
s_v=N\pi(v)\in\mathbb N,
\qquad
m_e=N w_e\in\mathbb N.
\]

\begin{figure}[H]
\centering
\includegraphics[page=1]{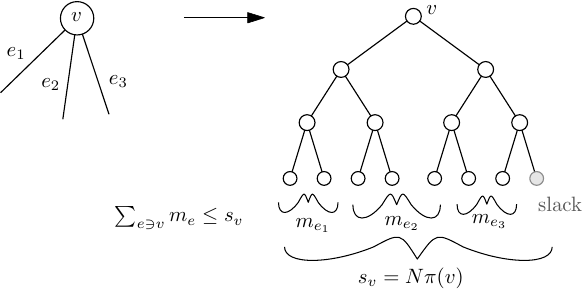}
\caption{Replacing vertex with a tree.}
\label{fig:vertexreplacement}
\end{figure}

By admissibility, we have $\sum_{e\ni v}m_e\le s_v$. We now replace each vertex $v$ of $G$ by the tree $T_v$ from Lemma~\ref{lem:tree} with $s_v$ ordered leaves, drawn in a small neighborhood of each vertex, see~\cref{fig:vertexreplacement}. These trees are connected to one another, using blocks of edges, according to how $G$ is connected. More precisely, for each edge $e$ in $G$ incident to $v$, we assign consecutive (in the cyclic order around $v$) blocks of $m_e$ leaves of $T_v$. For each such edge $e=uv$, we match a block associated to $v$ to a block associated to $u$ by inserting $m_e$ edges in order (reversing one order so that the edges realizing the matching are noncrossing). \cref{fig:edge-bundle} illustrates the procedure. Since $G$ was embedded in $\Sigma$, the resulting graph $H$ is too, has maximum degree at most three, and is simple and connected because the endpoints of the new edges are distinct and $w_e>0$ for every $e$. We also record
\begin{equation}\label{eq:size-H}
|V(H)|=\sum_v(2s_v-1)=2N-|V(G)|.
\end{equation}

\begin{figure}[H]
\centering
\includegraphics[page=2]{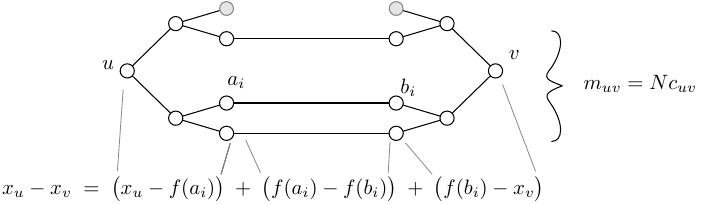}
\caption{A bundle replacing the edge $uv$. The term $x_u-x_v$ is compared against one added edge and two deviations from the averages of the inserted tree.}
\label{fig:edge-bundle}
\end{figure}

For $f:V(H)\to\RR$, we let $x_v$ be the average of $f$ on the tree replacing $v$. Lemma~\ref{lem:tree} and $|V(T_v)|<2N\pi(v)\le2N\|\pi\|_\infty$ imply, since we only sum over internal edges of the tree gadgets, that 
\begin{equation}\label{eq:total-variance}
\sum_v\sum_{z\in V(T_v)}(f(z)-x_v)^2
\le8N\|\pi\|_\infty\sum_{ab\in E(H)}(f(a)-f(b))^2.
\end{equation}
Now, $|V(T_v)|\le2s_v$, so a variance decomposition yields
\begin{align}\label{eq:denominator-transfer}
N\sum_v\pi(v)x_v^2
& \ge 
\sum_v \frac 12 |V(T_v)|x_v^2= 
\frac12 \sum_v \left(\sum_{z\in V(T_v)}f(z)^2- \sum_{z\in V(T_v)} (f(z)-x_v)^2 \right) \nonumber  \\
& \ge \frac 12 \sum_{z\in V(H)} f(z)^2 - 4N\|\pi\|_\infty\sum_{ab\in E(H)}(f(a)-f(b))^2.
\end{align}

For an edge $e=uv$ of the original graph $G$, we write its $m_e$ new edges as $a_i b_i$, with $a_i\in T_u$ and $b_i\in T_v$. Then we can reroute via the averages and use Cauchy-Schwarz again to give 
\[
m_e(x_u-x_v)^2
\le3\sum_{i=1}^{m_e}
\bigl((x_u-f(a_i))^2+(f(a_i)-f(b_i))^2+(f(b_i)-x_v)^2\bigr).
\]
Summing over $e$, every leaf of every inserted tree is used at most once, so applying \eqref{eq:total-variance} together with $N\|\pi\|_\infty\ge1$ yields 
\begin{equation}\label{eq:numerator-transfer}
\sum_{uv\in E(G)}w_{uv}(x_u-x_v)^2
\le 27\|\pi\|_\infty\sum_{ab\in E(H)}(f(a)-f(b))^2.
\end{equation}
We thus have an inequality for both the denominator and the numerator of the Rayleigh quotient, which we will use to conclude the proof. 
First, observe that since $N=\sum_v s_v\ge |V(G)|\ge k$, we get $|V(H)|\ge k$ from~\eqref{eq:size-H}. Let $U$ be the span of the first $k$ eigenvectors of $L_H$.  By hypothesis,
\[
\sum_{ab\in E(H)}(f(a)-f(b))^2
\le \frac{A}{|V(H)|}\sum_{z\in V(H)}f(z)^2,
\qquad \text{ for }f\in U.
\]
We have $|V(H)|\ge N$ and $A\|\pi\|_\infty<1/16$, so~\eqref{eq:denominator-transfer} shows 
\[
N\sum_v\pi(v)x_v^2\ge\frac14\sum_zf(z)^2,
\]
which implies that the linear map $f\mapsto x\in \RR^{V(G)}$ recording the average of $f$ on the tree replacing each vertex in $G$ is injective on $U$. Its image is thus a $k$-dimensional subspace of $\RR^{V(G)}$, and \eqref{eq:numerator-transfer} gives
\[
\frac{\sum_{uv\in E(G)}w_{uv}(x_u-x_v)^2}
     {\sum_v\pi(v)x_v^2}
\le108A\|\pi\|_\infty\frac{N}{|V(H)|}
\le108A\|\pi\|_\infty.
\]
The Courant-Fischer theorem bounds $\lambda_k(w,\pi)$. This proves the bound for positive rational pairs.

For general positive $\pi$ and nonnegative $w$, choose a positive weighting $w^0$ with strict slack. Linear combinations $(1-\varepsilon)w+\varepsilon w^0$ are positive and also have strict slack, so they can be approximated by positive rational admissible pairs $(\pi_j,w_j)$.  The matrices $\operatorname{diag}(\pi_j)^{-1/2}L_{w_j}\operatorname{diag}(\pi_j)^{-1/2}$ are arbitrarily close to the corresponding matrix for $(\pi,(1-\varepsilon)w+\varepsilon w^0)$.  Since eigenvalues depend continuously on the matrix, letting $j\to\infty$ gives the bound for $(\pi,(1-\varepsilon)w+\varepsilon w^0)$. Further sending $\varepsilon\to 0$ and noting that $w$ was arbitrary, we can finally maximize over all admissible $w$ to complete the proof.
\end{proof}

\section{Minor-Free Graphs}

In this section, we will prove
\cref{thm:main-minor}.
Our proof will leverage an existing pipeline summarized by
\cref{fig:minor-proof-architecture},
along with two new steps.
The first new step is fairly straightforward:
We use a black-box class-wide
lift from bounds on the uniform $\pi$-weight case
to bounds on the general $\pi$-weight case.
Our main novelty lies in the second new step:
An $L_2$-congestion bound for multicommodity flows based on the
Lov\'asz local lemma argument of
Korhonen and Lokshtanov~\cite{korhonen2024induced}.

The other components of our proof (the ``pipeline'') follow a line of reductions
initially laid out by Biswal, Lee, and Rao~\cite{biswal2008eigenvalue,biswal2010eigenvalue},
with related extensions by
Kelner, Lee, Price, and Teng~\cite{kelner2011metric},
Tung~\cite{tung2025reweighted},
and Spalding-Jamieson~\cite{spalding2025reweighted}.

\begin{figure}[H]
\centering
\resizebox{\textwidth}{!}{%
\begin{tikzpicture}[
  >={Stealth[length=2.1mm,width=1.45mm]},
  every node/.style={
    font=\scriptsize,
    text=figureink,
    align=center
  },
  card/.style={
    rounded corners=4pt,
    line width=0.9pt,
    minimum height=1.05cm,
    inner sep=5pt
  },
  routingquantity/.style={
    card,
    draw=routeorange!88!black,
    fill=routeorange!12
  },
  metricquantity/.style={
    card,
    draw=metricteal!82!black,
    fill=metricteal!8
  },
  onedquantity/.style={
    card,
    draw=localpurple!82!black,
    fill=localpurple!8
  },
  spectralquantity/.style={
    card,
    draw=resultblue!88!black,
    fill=resultblue!9
  },
  proofarrow/.style={
    ->,
    draw=figureink!68,
    line width=0.9pt
  },
  dualarrow/.style={
    <->,
    draw=figureink!68,
    line width=0.9pt
  },
  reductionarrow/.style={
    proofarrow,
    dashed
  },
  method/.style={
    font=\scriptsize\itshape,
    inner sep=0pt,
    text=figureink!82
  }
]
  \path[use as bounding box] (-10,-0.55) rectangle (12.75,12.55);

  \draw[
    rounded corners=5pt,
    draw=figurepanel!95,
    fill=figurepanel!18,
    line width=0.8pt
  ] (-9.75,-0.3) rectangle (12.5,12.3);
  \node[
    font=\scriptsize\bfseries,
    anchor=west
  ] at (-9.45,11.95) {
    Reweighted eigenvalues for $2\leq k\leq n/8$,
    where $r=\lfloor n/(4k)\rfloor$
  };

  \draw[
    rounded corners=5pt,
    draw=figureink!55,
    fill=figurepanel!38,
    dash pattern=on 0pt off 2.2pt,
    line cap=round,
    line width=0.8pt
  ] (-9.35,-0.05) rectangle (6.25,4.38);
  \draw[draw=figurepanel!38,line width=2pt]
    (-7.72,4.38)--(-7.28,4.38);
  \draw[draw=figurepanel!38,line width=2pt]
    (3.78,4.38)--(4.22,4.38);
  \node[
    method,
    fill=figurepanel!38,
    rounded corners=3pt,
    inner xsep=4pt,
    inner ysep=1pt
  ] at (-1.55,4.38) {
    existing pipeline
  };

  \node[
    routingquantity,
    text width=3.8cm,
    inner xsep=3pt
  ] (productbound) at (-7.5,9.75) {
    product-demand $L_2$ bound\\
    $\displaystyle\con_2(F)\gtrsim\alpha/h$
  };
  \node[
    routingquantity,
    text width=4cm,
    inner xsep=3pt
  ] (demandbound) at (-7.5,7.65) {
    unit $H$-flow $L_2$ congestion\\
    $\displaystyle\con_2(F)\gtrsim
      \frac{|E(H)|^{3/2}}{h|V(H)|}$
  };

  \node[method] at (-7.5,10.8) {
    Lov\'asz local lemma
  };

  \node[
    routingquantity,
    text width=3.2cm
  ] (subsetflow) at (-7.5,5.3) {
    normalized $r$-subset $L_2$ congestion\\
    $\displaystyle\gtrsim\sqrt{r/(h^2n)}$
  };
  \node[
    metricquantity,
    text width=3cm
  ] (subsetmetric) at (-7.5,3.1) {
    $r$-subset metric spread\\
    $\displaystyle\gtrsim\sqrt{r/(h^2n)}$
  };
  \node[
    metricquantity,
    text width=3.2cm
  ] (squaredmetric) at (-1.5,3.1) {
    squared $r$-subset metric spread\\
    $\displaystyle\gtrsim r/(h^2n)$
  };
  \node[
    onedquantity,
    text width=3.8cm
  ] (onedspread) at (-1.5,0.65) {
    one-dimensional squared spread\\
    $\displaystyle\gtrsim r/(h^2\log^2h)$
  };
  \node[
    onedquantity,
    text width=3.7cm
  ] (onedcover) at (4,0.65) {
    one-dimensional dual form\\
    $\displaystyle\lesssim h^2\log^2h/r$
  };
  \node[
    spectralquantity,
    text width=3.8cm
  ] (fullcover) at (4,3.1) {
    full-dimensional dual form\\
    $\displaystyle\lesssim h^2\log^2h/r
      \lesssim h^2k\log^2h/n$
  };
  \node[
    spectralquantity,
    text width=3.1cm,
    inner xsep=2pt,
    font=\scriptsize\bfseries
  ] (uniformbound) at (4,5.3) {
    $\lambda_k^*(G)\lesssim h^2k\log^2h/n$
  };
  \node[
    spectralquantity,
    text width=3.55cm,
    inner xsep=2pt,
    font=\scriptsize\bfseries
  ] (weightedbound) at (10.4,5.3) {
    $\lambda_k^*(G,\pi)\lesssim h^2k\log^2h\|\pi\|_\infty$
  };

  \draw[proofarrow] (productbound)--node[
    method,
    right=5pt
  ] {
    decompose $H$ into expanding subgraphs
  } (demandbound);
  \draw[proofarrow] (demandbound)--node[
    method,
    right=5pt
  ] {
    thresholding
  } (subsetflow);

  \draw[dualarrow] (subsetflow)--node[
    method,
    right=5pt,
    yshift=-3pt
  ] {
    strong duality
  } (subsetmetric);
  \draw[proofarrow] (subsetmetric)--node[
    method,
    above=8pt
  ] {
    Cauchy-Schwarz
  } (squaredmetric);
  \draw[proofarrow] (squaredmetric)--node[
    method,
    right=5pt
  ] {
    padded decompositions
  } (onedspread);
  \draw[dualarrow] (onedspread)--node[
    method,
    above=5pt
  ] {
    reciprocal\\identity
  } (onedcover);
  \draw[proofarrow] (onedcover)--node[
    method,
    left=5pt
  ] {
    relaxation
  } (fullcover);
  \draw[dualarrow] (fullcover)--node[
    method,
    left=5pt,
    yshift=-3pt
  ] {
    strong duality
  } (uniformbound);
  \draw[reductionarrow]
    (uniformbound)--node[
      method,
      above=6pt
    ] {
      vertex splitting
    } (weightedbound);
\end{tikzpicture}
}
\caption{
Proof architecture for an $n$-vertex $K_h$-minor-free graph.
}
\label{fig:minor-proof-architecture}
\end{figure}
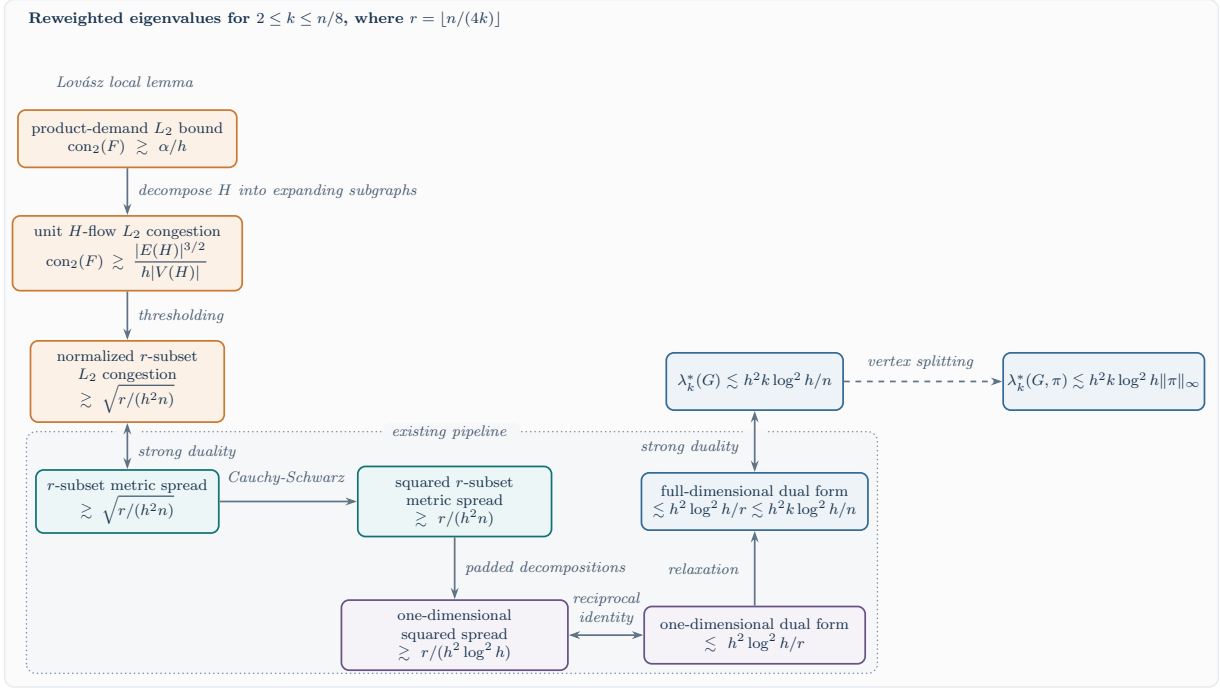

\begin{definition}[{\cite{biswal2008eigenvalue,biswal2010eigenvalue,kelner2011metric}}]
Let $G=(V,E)$ be a graph.
For $u,v\in V$,
let $\mathcal P_{uv}$ be the set of all $u$--$v$ paths in $G$,
and let $\mathcal P$ be the set of all paths in $G$.
A \defn{multicommodity flow} in $G$ is a function
$F\colon\mathcal P\to\RR_{\geq0}$.
Write $\mathcal F(G)$ for the set of multicommodity flows in $G$ and denote
\[
F[u,v]
:=
\sum_{p\in\mathcal P_{uv}}F(p).
\]
For a graph $H$ with $V(H)\subset V$,
a \defn{unit $H$-flow} satisfies $F[u,v]=1$ if $\{u,v\}\in E(H)$
and $F[u,v]=0$ otherwise.
Define the congestion of $F$ at a vertex $x\in V$ as
\[
c_F(x)
=
\sum_{p\in\mathcal P:x\in p}F(p).
\]
For an edge $e\in E$,
write
\[
\ell_F(e)
=
\sum_{p\in\mathcal P:e\in p}F(p)
\]
for its load under $F$.
We define the \defn{$L_2$-congestion} of $F$
as
\[
\con_2(F)
=
\left(\sum_{x\in V}c_F(x)^2\right)^{1/2}.
\]
Finally,
define the
\defn{extremal $L_2$-congestion} of $G$
as the minimum of $\con_2(F)$ over all unit $K_V$-flows $F$ in $G$.
We denote this minimum by $\con_2(G)$.
\end{definition}

\begin{proposition}
\label{prop:minor-congestion}
Let $G=(V,E)$ be a $K_h$-minor-free graph with $h\geq3$,
let $p$ be a probability distribution on $V$ with $\|p\|_\infty\leq\frac12$,
and let $\alpha>0$.
Every flow $F$ in $G$ satisfying
$F[u,v]=\alpha p(u)p(v)$ for every $\{u,v\}\in\binom V2$ satisfies
\[
\con_2(F)
\gtrsim
\frac{\alpha}{h}.
\]
\end{proposition}

The proof of this step is precisely where we will apply
the Lov\'asz local lemma.
We defer the proof to \cref{subsec:minor-congestion}.

\begin{definition}[{\cite[Lemma~3.5.17]{tung2025reweighted}}]
\label{def:subset-congestion}
Let $G=(V,E)$ be a graph,
and let $2\leq r\leq|V|$.
The \defn{normalized $r$-subset congestion} is
\[
\begin{array}{rccr@{\:}c@{\:}ll}
\con_{2,r}(G)
&=&
\displaystyle\min_{\substack{
  \mu\in\RR_{\geq0}^{\binom Vr}\\
  F\in\mathcal F(G)
}}
&\multicolumn{4}{l}{\con_2(F)}\\
&&\text{subject to}
&\displaystyle\sum_{A\in\binom Vr}\mu(A)
&{}={}&1,\\
&&
&F[u,v]
&{}={}&\displaystyle
\frac1{r^2}
\sum_{\substack{A\in\binom Vr\\u,v\in A}}\mu(A)
&\forall\{u,v\}\in\binom V2.
\end{array}
\]
\end{definition}

We will use
\cref{prop:minor-congestion}
and some ``lifting'' arguments to prove the following bound.

\begin{theorem}
\label{thm:minor-subset-flow}
Let $G$ be an $n$-vertex $K_h$-minor-free graph for $h\geq3$,
and let $2\leq r\leq n$.
Then
\[
\con_{2,r}(G)^2
\gtrsim
\frac{r}{h^2n}.
\]
\end{theorem}

Once again, we defer the proof to \cref{subsec:minor-congestion}.

\begin{lemma}[{\cite[Lemmas~3.5.17 and~7.4.3]{tung2025reweighted}%
    +\cite[Theorem~2]{conroy2025protect}}]
\label{lem:minor-flow-metric}
Let $G$ be an $n$-vertex $K_h$-minor-free graph for $h\geq3$,
let $2\leq k\leq n/8$,
and set $r=\lfloor n/(4k)\rfloor$.
Then
\[
\lambda_k^*(G)
\lesssim
\frac{\log^2h}{n\cdot\con_{2,r}(G)^2}.
\]
\end{lemma}

\begin{corollary}
\label{cor:minor-uniform-eigenvalues}
Every $n$-vertex $K_h$-minor-free graph with $h\geq3$ satisfies
\[
\lambda_k^*(G)
\lesssim
\frac{h^2k\log^2h}{n}.
\]
This holds for every $2\leq k\leq n$.
\end{corollary}

\begin{proof}
For $k>n/8$,
the claim follows from $\lambda_k^*(G)\leq2$.
Otherwise,
$r=\lfloor n/(4k)\rfloor\geq n/(8k)$,
and the result follows from
\cref{thm:minor-subset-flow,lem:minor-flow-metric}.
\end{proof}

In order to lift this to the general $\pi$-weighted result,
we will use a simple operation.

\begin{lemma}
\label{lem:graph-family-closure}
For every $g\geq0$ and every graph $H$ with minimum degree at least $2$,
the families of graphs of orientable genus at most $g$ and of $H$-minor-free
graphs are closed under attaching leaves.
\end{lemma}

\begin{proof}
For the first family,
start with a surface embedding,
and draw the new leaves near their incident vertices.

For the second family,
deleting the added leaves from any $H$-minor model gives an $H$-minor model
in the original graph because $H$ has minimum degree at least two.
\end{proof}

The operation in \cref{lem:graph-family-closure} is shown in
\cref{fig:graph-family-closure}.

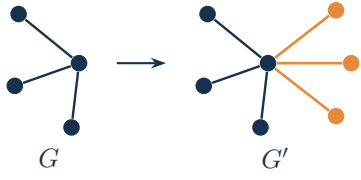
\begin{figure}[H]
\centering
\begin{tikzpicture}[
  every node/.style={font=\small,text=figureink},
  vertex/.style={circle,fill=figureink,inner sep=2.2pt},
  newvertex/.style={circle,fill=routeorange,inner sep=2.2pt},
  edge/.style={draw=figureink,line width=0.9pt},
  newedge/.style={draw=routeorange,line width=1pt},
  transform/.style={-{Stealth[length=2mm,width=1.4mm]},draw=figureink,line width=0.8pt}
]
  \begin{scope}[xshift=1.5mm]
  \node[vertex] (leftv) at (-5.25,0) {};
  \node[vertex] (lefta) at (-6.05,0.65) {};
  \node[vertex] (leftb) at (-6.1,-0.3) {};
  \node[vertex] (leftc) at (-5.35,-0.85) {};
  \draw[edge] (leftv)--(lefta);
  \draw[edge] (leftv)--(leftb);
  \draw[edge] (leftv)--(leftc);
  \node at (-5.65,-1.25) {$G$};

  \draw[transform] (-4.75,0)--(-4.1,0);

  \node[vertex] (rightv) at (-2.75,0) {};
  \node[vertex] (righta) at (-3.55,0.65) {};
  \node[vertex] (rightb) at (-3.6,-0.3) {};
  \node[vertex] (rightc) at (-2.85,-0.85) {};
  \node[newvertex] (leafone) at (-1.85,0.7) {};
  \node[newvertex] (leaftwo) at (-1.65,0) {};
  \node[newvertex] (leafthree) at (-1.85,-0.7) {};
  \draw[edge] (rightv)--(righta);
  \draw[edge] (rightv)--(rightb);
  \draw[edge] (rightv)--(rightc);
  \draw[newedge] (rightv)--(leafone);
  \draw[newedge] (rightv)--(leaftwo);
  \draw[newedge] (rightv)--(leafthree);
  \node at (-2.65,-1.25) {$G'$};
  \end{scope}
\end{tikzpicture}
\caption{Attaching leaves.}
\label{fig:graph-family-closure}
\end{figure}

To extend this to arbitrary vertex weights,
we use leaf-attaching to prove the following lemma.

\begin{lemma}
\label{lem:vertex-splitting}
Let $G=(V,E)$ be a graph,
let $\pi$ be a positive probability distribution on $V$,
and let $2\leq k\leq|V|$.
There is an $N$-vertex graph $G'$,
for some positive integer $N$,
such that
\[
\lambda_k^*(G,\pi)
\lesssim
N\|\pi\|_\infty\lambda_k^*(G').
\]
Moreover,
$G'$ is $K_h$-minor-free whenever $G$ is $K_h$-minor-free and $h\geq3$.
\end{lemma}

\begin{proof}
Choose a positive rational probability distribution $\widehat\pi$ satisfying
$\frac12\pi(u)\leq\widehat\pi(u)\leq2\pi(u)$ for every $u\in V$.
Then
$\lambda_k^*(G,\pi)\lesssim\lambda_k^*(G,\widehat\pi)$.

Let $N$ be a common denominator of the values of $\widehat\pi$,
and construct $G'$ from $G$ by attaching $N\widehat\pi(u)-1$ leaves to each $u$.
By \cref{lem:graph-family-closure},
$G'$ is $K_h$-minor-free whenever $G$ is $K_h$-minor-free and $h\geq3$.
Hence, we need only show the inequality
$\lambda_k^*(G,\widehat\pi)\lesssim
N\|\widehat\pi\|_\infty\cdot\lambda_k^*(G')$.

For each $u\in V$,
let $C_u\subset V(G')$ consist of $u$ together with its added leaves,
so $|C_u|=N\widehat\pi(u)$.
Write $\widehat\Pi=\operatorname{diag}(\widehat\pi)$.
Fix an edge weighting $w$ feasible in the optimization problem in
\cref{obs:reweighted-weighted-laplacian} for $(G,\widehat\pi)$.
Define an edge weighting $\widetilde w$ on $G'$ by
\[
\begin{aligned}
\widetilde w_{uv}
&:=
\frac{w_{uv}}{2N\|\widehat\pi\|_\infty}
&& (uv\in E),\\
\widetilde w_{uz}
&:=
\frac1{2N|C_u|}
&& (u\in V, z\in C_u\setminus\{u\}).
\end{aligned}
\]
For each original vertex $u\in V$,
the sum of incident edge weights is at most
$\widehat\pi(u)/(2N\|\widehat\pi\|_\infty)
+(|C_u|-1)/(2N|C_u|)<1/N$,
while for each added vertex $z\in C_u\setminus\{u\}$,
there is a unique incident edge weight $1/(2N|C_u|)<1/N$.
Hence, $\widetilde w$ is feasible in the optimization problem in
\cref{obs:reweighted-weighted-laplacian} for $G'$ with uniform weights.
Since $w$ was arbitrary,
and $\lambda_k^*(G')\geq\lambda_k(NL_{\widetilde w})$,
it suffices to show that
\[
\lambda_k\left(\widehat\Pi^{-1/2}L_w\widehat\Pi^{-1/2}\right)
\lesssim
N\|\widehat\pi\|_\infty\cdot
\lambda_k\left(NL_{\widetilde w}\right).
\]

Let $\mu:=\lambda_k(NL_{\widetilde w})$.
Since the feasibility of $w$ gives
$\lambda_k(\widehat\Pi^{-1/2}L_w\widehat\Pi^{-1/2})\leq2$,
the desired bound holds if
$\mu\geq1/(4N\|\widehat\pi\|_\infty)$.
Assume otherwise.
By Courant-Fischer,
there is a $k$-dimensional subspace $U$ of $\RR^{V(G')}$ such that
every $y\in U$ satisfies
\[
\sum_{ab\in E(G')}\widetilde w_{ab}(y_a-y_b)^2
\leq
\frac\mu N\sum_{z\in V(G')}y_z^2.
\]

Let $y\in U$ be arbitrary.
Let $x_u$ be the mean of the values $y_z$ over $z\in C_u$,
so $x_u$ minimizes $\sum_{z\in C_u}(y_z-x_u)^2$.
Since the sets $C_u$ partition $V(G')$ and
$|C_u|=N\widehat\pi(u)$,
we have
\[
\begin{aligned}
\frac1N\sum_{z\in V(G')}y_z^2
-\sum_{u\in V}\widehat\pi(u)x_u^2
&=
\frac1N\sum_{u\in V}
\left(
\left(\sum_{z\in C_u}y_z^2\right)-|C_u|x_u^2
\right)\\
&=
\frac1N\sum_{u\in V}\sum_{z\in C_u}(y_z-x_u)^2\\
&\leq
\frac1N\sum_{u\in V}\sum_{z\in C_u\setminus\{u\}}(y_z-y_u)^2\\
&=
2\sum_{u\in V}|C_u|
\sum_{z\in C_u\setminus\{u\}}
\widetilde w_{uz}(y_u-y_z)^2\\
&\leq
2N\|\widehat\pi\|_\infty\cdot
\sum_{ab\in E(G')}\widetilde w_{ab}(y_a-y_b)^2.
\end{aligned}
\]
For each $u\in V$,
the definition of $x_u$ gives
\[
x_u-y_u
=
\frac1{|C_u|}
\sum_{z\in C_u\setminus\{u\}}(y_z-y_u).
\]
Since $\widehat\pi(u)=|C_u|/N$,
Cauchy-Schwarz gives
\[
\begin{aligned}
\widehat\pi(u)(x_u-y_u)^2
&=
\frac1{N|C_u|}
\left(
\sum_{z\in C_u\setminus\{u\}}(y_z-y_u)
\right)^2\\
&\leq
\frac1N\sum_{z\in C_u\setminus\{u\}}(y_z-y_u)^2.
\end{aligned}
\]
Summing over $u$ and using the definition of $\widetilde w$ gives
\[
\begin{aligned}
\sum_{u\in V}\widehat\pi(u)(x_u-y_u)^2
&\leq
\frac1N\sum_{u\in V}\sum_{z\in C_u\setminus\{u\}}(y_z-y_u)^2\\
&\leq
2N\|\widehat\pi\|_\infty\cdot
\sum_{u\in V}\sum_{z\in C_u\setminus\{u\}}
\widetilde w_{uz}(y_u-y_z)^2.
\end{aligned}
\]
For every $uv\in E$,
Cauchy-Schwarz gives
\[
\begin{aligned}
(x_u-x_v)^2
&=
\left((x_u-y_u)+(y_u-y_v)+(y_v-x_v)\right)^2\\
&\leq
3(x_u-y_u)^2+3(y_u-y_v)^2+3(x_v-y_v)^2.
\end{aligned}
\]
Multiplying by $w_{uv}$,
summing over $E$,
and using the feasibility of $w$ gives
\[
\begin{aligned}
\sum_{uv\in E}w_{uv}(x_u-x_v)^2
&\leq
3\sum_{uv\in E}w_{uv}(y_u-y_v)^2
+
3\sum_{u\in V}
\left(\sum_{e\ni u}w_e\right)(x_u-y_u)^2\\
&\leq
3\sum_{uv\in E}w_{uv}(y_u-y_v)^2
+
3\sum_{u\in V}\widehat\pi(u)(x_u-y_u)^2\\
&=
6N\|\widehat\pi\|_\infty\cdot
\sum_{uv\in E}\widetilde w_{uv}(y_u-y_v)^2
+
3\sum_{u\in V}\widehat\pi(u)(x_u-y_u)^2\\
&\leq
6N\|\widehat\pi\|_\infty\cdot
\left(
\sum_{uv\in E}\widetilde w_{uv}(y_u-y_v)^2
+
\sum_{u\in V}\sum_{z\in C_u\setminus\{u\}}
\widetilde w_{uz}(y_u-y_z)^2
\right)\\
&=
6N\|\widehat\pi\|_\infty\cdot
\sum_{ab\in E(G')}\widetilde w_{ab}(y_a-y_b)^2.
\end{aligned}
\]
Since $y\in U$ and $\mu<1/(4N\|\widehat\pi\|_\infty)$,
\[
\begin{aligned}
\frac1N\sum_{z\in V(G')}y_z^2
-\sum_{u\in V}\widehat\pi(u)x_u^2
&\leq
2N\|\widehat\pi\|_\infty\cdot
\sum_{ab\in E(G')}\widetilde w_{ab}(y_a-y_b)^2\\
&\leq
2N\|\widehat\pi\|_\infty\cdot\mu
\left(
\frac1N\sum_{z\in V(G')}y_z^2
\right)\\
&<
\frac12
\left(
\frac1N\sum_{z\in V(G')}y_z^2
\right).
\end{aligned}
\]
Finally,
we apply Courant-Fischer to get
\[
\begin{aligned}
\lambda_k\left(\widehat\Pi^{-1/2}L_w\widehat\Pi^{-1/2}\right)
&\leq
\max_{\substack{y\in U\\y\neq0}}
\frac{
  \sum_{uv\in E}w_{uv}(x_u-x_v)^2
}{
  \sum_{u\in V}\widehat\pi(u)x_u^2
}\\
&<
12N\|\widehat\pi\|_\infty\cdot
\max_{\substack{y\in U\\y\neq0}}
\frac{
  \sum_{ab\in E(G')}\widetilde w_{ab}(y_a-y_b)^2
}{
  \frac1N\sum_{z\in V(G')}y_z^2
}\\
&\leq
12N\|\widehat\pi\|_\infty\cdot\mu.
\end{aligned}
\]
\end{proof}

We can now prove the main result of this section,
save for the deferred proof of \cref{thm:minor-subset-flow}.

\begin{proof}[Proof of \cref{thm:main-minor}]
Apply \cref{lem:vertex-splitting,cor:minor-uniform-eigenvalues} to obtain
\[
\lambda_k^*(G,\pi)
\lesssim
N\|\pi\|_\infty\lambda_k^*(G')
\lesssim
h^2k\log^2h\,\|\pi\|_\infty.
\]
\end{proof}

\subsection{A Proof of the Congestion Bound}
\label{subsec:minor-congestion}

We will make use of the following variant of a definition
by Korhonen and Lokshtanov~\cite[Section~4.2]{korhonen2024induced}
(see also \cite[Section 4]{le2026separatorminorfreegraphsflow}).
\begin{definition}
An \defn{almost-embedding} of a graph $H$ into a graph $G$
maps each vertex $x\in V(H)$ to a vertex $\phi(x)\in V(G)$
and each edge $e=xy\in E(H)$ to a $\phi(x)$--$\phi(y)$ path $\phi(e)$ in $G$,
such that if $e$ and $f$ have no common endpoint,
then $\phi(e)$ and $\phi(f)$ are vertex-disjoint.
\end{definition}

Korhonen and Lokshtanov essentially prove a stronger form of the following result.

\begin{lemma}[{\cite[proofs of Lemmas~4.3 and~4.4]{korhonen2024induced}}]
\label{lem:clique-almost-embedding}
For every $h\geq3$,
there is a graph $H$ of minimum degree at least two
and maximum degree at most three with $|E(H)|\leq3h^2$
such that every graph into which $H$ almost-embeds contains $K_h$ as a minor.
\end{lemma}

\begin{proof}[Proof of \cref{prop:minor-congestion}]
Let $G=(V,E)$,
and let $p$, $\alpha$, and $F$ be as in the proposition.
Let $H$ be given by \cref{lem:clique-almost-embedding},
and write $m_H=|E(H)|\leq3h^2$.

For every $u\in V$,
\[
c_F(u)
\geq
\alpha p(u)(1-p(u))
\geq
\frac{\alpha}2p(u).
\]
Consequently,
$\|p\|_2\leq2\con_2(F)/\alpha$.

We will construct a random ``embedding'' $\phi$ of $H$ into $G$.
Independently map every vertex $x\in V(H)$ to $u\in V$
with probability $p(u)$.
For each edge $e=xy\in E(H)$,
independently sample a path $P\in\mathcal P_{\phi(x)\phi(y)}$
with probability $F(P)/(\alpha p(\phi(x))p(\phi(y)))$
when $\phi(x)\neq\phi(y)$,
and take the one-vertex path when $\phi(x)=\phi(y)$.
Denote the resulting path by $\phi(e)$.

For $v\in V(G)$,
the probability that the path sampled for any fixed edge of $H$ contains $v$ is
\[
\begin{aligned}
q(v)
&=
2
\sum_{\substack{\{u,w\}\in\binom V2\\p(u)p(w)>0}}
p(u)p(w)
\sum_{\substack{P\in\mathcal P_{uw}\\v\in P}}
\frac{F(P)}{\alpha p(u)p(w)}
+
p(v)^2\\
&=
\frac{2c_F(v)}{\alpha}+p(v)^2.
\end{aligned}
\]
Consequently,
\[
\|q\|_2
\leq
\frac{2\con_2(F)}{\alpha}
+
\|p\|_2^2.
\]
If $e,f\in E(H)$ have no common endpoint,
then $\phi(e)$ and $\phi(f)$ are independent.
Thus
\[
\Pr[\phi(e)\cap\phi(f)\neq\emptyset]
\leq
\sum_{v\in V(G)}q(v)^2
=
\|q\|_2^2.
\]

Suppose for a contradiction that
\[
\con_2(F)
<
\varepsilon\frac{\alpha}{h},
\]
for a small constant $\varepsilon$ that we will choose shortly.
Then
\[
\Pr[\phi(e)\cap\phi(f)\neq\emptyset]
\leq
\left(
\frac{2\varepsilon}{h}
+
\frac{4\varepsilon^2}{h^2}
\right)^2.
\]

The event $\phi(e)\cap\phi(f)\neq\emptyset$
depends only on the images of the endpoints of $e$ and $f$
and on the paths sampled for $e$ and $f$.
Since $H$ has maximum degree at most three,
at most $12$ edges of $H$ are incident to an endpoint of $e$ or $f$.
An intersection event for a different pair of nonincident edges is independent
of this event unless one of its edges is among these $12$ edges.
For such an event,
there are at most $m_H$ choices for its second edge,
so the event $\phi(e)\cap\phi(f)\neq\emptyset$
depends on at most $12m_H$ other intersection events.

Take $\varepsilon=1/100$.
Since $h\geq3$ and $m_H\leq3h^2$,
\[
\begin{aligned}
\mathrm e(12m_H+1)
\left(
\frac{2\varepsilon}{h}
+
\frac{4\varepsilon^2}{h^2}
\right)^2
&\leq
37\mathrm e h^2
\left(
\frac1{50h}
+
\frac1{2500h^2}
\right)^2\\
&\leq
37\mathrm e
\left(
\frac1{50}
+
\frac1{7500}
\right)^2\\
&<
1.
\end{aligned}
\]
The Lov\'asz local lemma therefore proves
that $\phi$ is an almost-embedding of $H$ into $G$ with nonzero probability.
By \cref{lem:clique-almost-embedding},
$K_h$ is a minor of $G$,
a contradiction.
\end{proof}

This ``all-pairs flow'' result would already be enough if we only wanted a bound on $\lambda_2^*$.
In order to obtain our bound for higher reweighted eigenvalues,
we will need to generalize this result to a wider variety of flows
satisfying the constraints in \cref{def:subset-congestion}.
This generalization will be done by carefully partitioning the flow endpoints (the demand graph)
using some existing machinery.

\begin{lemma}[{\cite[Theorem~3.1]{FeigeHL08}}]
\label{lem:weighted-bourgain}
Let $(X,d)$ be a finite metric space,
and let $q$ be a probability distribution on $X$.
There is a function $f\colon X\to\RR$ such that
\[
\sum_{x,y\in X}q(x)q(y)d(x,y)
\lesssim
\log(|X|)\cdot
\sum_{x,y\in X}q(x)q(y)|f(x)-f(y)|,
\]
and,
for every $x,y\in X$,
$|f(x)-f(y)|\leq d(x,y)$.
\end{lemma}

\begin{lemma}[{\cite[Theorem~3.1]{frank2002node}}]
\label{lem:edge-flow-distance}
Let $G=(V,E)$ be a connected graph,
let $c\colon E\to\RR_{\geq0}$,
and let $D\colon\binom V2\to\RR_{\geq0}$.
There is a flow $F$ in $G$ satisfying
$F[u,v]=D(u,v)$ for every $\{u,v\}\in\binom V2$ and
$\ell_F(e)\leq c(e)$ for every $e\in E$
if and only if every $\ell\colon E\to\RR_{\geq0}$ satisfies
\[
\sum_{e\in E}c(e)\ell(e)
\geq
\sum_{\{u,v\}\in\binom V2}D(u,v)d_\ell(u,v),
\]
where $d_\ell$ is the shortest-path metric induced by $\ell$.
\end{lemma}

For a graph $G$ and $S\subset V(G)$,
we say the \defn{volume} of $S$ is
$\operatorname{vol}_G(S):=\sum_{u\in S}\deg_G(u)$,
and we write $\delta_G(S)$ for the set of edges with exactly one endpoint in $S$.
We will use the external lemmas to prove the following helpful result.

\begin{proposition}
\label{prop:expansion-product-flow}
Let $G=(V,E)$ be a connected graph with $m$ edges
and minimum degree at least $d>0$.
Let $0<\phi\leq1$.
Assume
every $S\subset V$ satisfies
\begin{equation}
|\delta_G(S)|
\geq
\phi\min\{\operatorname{vol}_G(S),
\operatorname{vol}_G(V\setminus S)\}.
\label{eq:transfer-expansion}
\end{equation}
Then,
for some $\alpha\gtrsim\phi m/\log(m/d^2)$,
$G$ has a flow $F$ with $\ell_F(e)\leq1$ for every $e\in E$ and
\begin{equation}
F[u,v]
=
\frac{\alpha\deg_G(u)\deg_G(v)}{4m^2}
\qquad
\forall\{u,v\}\in\binom V2.
\label{eq:transfer-product-flow}
\end{equation}
\end{proposition}

\begin{proof}
Set $q(u):=\deg_G(u)/(2m)$ for every $u\in V$.
We first prove that every nonnegative edge-length function
$\ell\colon E\to\RR_{\geq0}$ satisfies
\begin{equation}
\sum_{\{u,v\}\in\binom V2}q(u)q(v)d_\ell(u,v)
\lesssim
\frac{\log(m/d^2)}{\phi m}
\sum_{e\in E}\ell(e).
\label{eq:expansion-average-distance}
\end{equation}
By perturbing and rescaling $\ell$,
we may assume without loss of generality that every edge has positive length
and $\sum_{e\in E}\ell(e)=1$.

For each $u\in V$,
let $B_u$ consist of $u$ and the other endpoints of the
$\lfloor\deg_G(u)/2\rfloor$ shortest edges incident to $u$.
Iterate over the vertices $u\in V$ in nondecreasing order of their average
incident edge length.
We will select a subset of the vertices as we iterate over them.
For each vertex $u\in V$,
if $B_u$ is disjoint from every $B_z$ with $z$ already selected,
select $u$ and set $\kappa(u):=u$.
Otherwise,
let $\kappa(u)$ be any previously selected $z$ such that
$B_u\cap B_z\neq\emptyset$.
Let $Z$ be the set of selected vertices.

The sets $B_z$ for $z\in Z$ are pairwise disjoint,
and each has size at least $d/2$.
Moreover,
$|V|\leq2m/d$,
so
\begin{equation}
|Z|
\leq
\frac{2|V|}{d}
\leq
\frac{4m}{d^2}.
\label{eq:center-count}
\end{equation}
For every edge $ux$ with $x\in B_u\setminus\{u\}$,
at least $\deg_G(u)/2$ edges incident to $u$ have length at least $\ell(ux)$.
Therefore we can relate $\ell(ux)$ to the average edge-length incident to $u$ with
\[
\begin{aligned}
\frac{\deg_G(u)}2\ell(ux)
&\leq
\sum_{e\ni u}\ell(e),\\
\ell(ux)
&\leq
\frac2{\deg_G(u)}
\sum_{e\ni u}\ell(e).
\end{aligned}
\]
For $u\notin Z$,
fix $x\in B_u\cap B_{\kappa(u)}$.
Since $\kappa(u)$ was considered before $u$,
its average incident edge length is at most that of $u$.
Noting that $x,u,\kappa(u)$ may not all be distinct,
the triangle inequality and the preceding bound give
\[
\begin{aligned}
d_\ell(u,\kappa(u))
&\leq
d_\ell(u,x)+d_\ell(x,\kappa(u))\\
&\leq
\frac2{\deg_G(u)}
\sum_{e\ni u}\ell(e)
+
\frac2{\deg_G(\kappa(u))}
\sum_{e\ni\kappa(u)}\ell(e)\\
&\leq
\frac4{\deg_G(u)}
\sum_{e\ni u}\ell(e).
\end{aligned}
\]
The same bound is immediate for $u\in Z$, where the left-hand-side evaluates to zero.
Consequently,
\begin{equation}
\begin{aligned}
\sum_{u\in V}q(u)d_\ell(u,\kappa(u))
&\leq
\sum_{u\in V}
\frac{\deg_G(u)}{2m}\cdot
\frac4{\deg_G(u)}
\sum_{e\ni u}\ell(e)\\
&=
\frac2m\sum_{u\in V}\sum_{e\ni u}\ell(e)\\
&=
\frac4m.
\end{aligned}
\label{eq:center-distance}
\end{equation}

Form a connected multigraph $G'$ on $Z$ by replacing each edge $uv\in E$
with an edge from $\kappa(u)$ to $\kappa(v)$ of length
\[
d_\ell(u,\kappa(u))
+
\ell(uv)
+
d_\ell(v,\kappa(v)).
\]
Let $\ell'$ be this edge-length function,
and let $d'$ be its shortest-path metric.
By \eqref{eq:center-distance},
\begin{equation}
\begin{aligned}
\sum_{e\in E(G')}\ell'(e)
&=
\sum_{uv\in E}
\left(
d_\ell(u,\kappa(u))
+
\ell(uv)
+
d_\ell(v,\kappa(v))
\right)\\
&=
1+
\sum_{u\in V}\deg_G(u)d_\ell(u,\kappa(u))\\
&=
1+
2m\sum_{u\in V}q(u)d_\ell(u,\kappa(u))\\
&\leq
9.
\end{aligned}
\label{eq:quotient-total-length}
\end{equation}
Set
$\widehat q(z):=\sum_{u:\kappa(u)=z}q(u)$.
For $Y\subset Z$,
write $\widehat q(Y):=\sum_{z\in Y}\widehat q(z)$.
The edges crossing $Y$ in $G'$ are precisely the images of the edges crossing
$\kappa^{-1}(Y)$ in $G$.
Hence,
by \eqref{eq:transfer-expansion},
\begin{equation}
\begin{aligned}
|\delta_{G'}(Y)|
&=
|\delta_G(\kappa^{-1}(Y))|\\
&\geq
2\phi\cdot m
\cdot\min\{\widehat q(Y),1-\widehat q(Y)\}\\
&\geq
2\phi\cdot m\cdot\widehat q(Y)(1-\widehat q(Y)).
\end{aligned}
\label{eq:quotient-expansion}
\end{equation}

Apply \cref{lem:weighted-bourgain} to $(Z,d')$ and $\widehat q$.
For the resulting function $f$,
set $Y_t:=\{z\in Z:f(z)\leq t\}$.
The identity
\[
|f(z)-f(z')|
=
\int_\RR
\left|\mathbf 1[z\in Y_t]-\mathbf 1[z'\in Y_t]\right|\,dt
\]
gives both summation-to-integral equalities below.
We obtain
\[
\begin{array}{rcll}
\displaystyle
\sum_{z,z'\in Z}\widehat q(z)\widehat q(z')d'(z,z')
&\lesssim&
\multicolumn{2}{l}{
\displaystyle
\log(|Z|)\cdot
\sum_{z,z'\in Z}\widehat q(z)\widehat q(z')|f(z)-f(z')|
}\\
&=&
\multicolumn{2}{l}{
\displaystyle
\log(|Z|)\cdot
\int_\RR
\sum_{z,z'\in Z}
\widehat q(z)\widehat q(z')
\left|\mathbf 1[z\in Y_t]-\mathbf 1[z'\in Y_t]\right|\,dt
}\\
&=&
\multicolumn{2}{l}{
\displaystyle
2\log(|Z|)\cdot
\int_\RR
\sum_{z\in Y_t}\sum_{z'\notin Y_t}
\widehat q(z)\widehat q(z')\,dt
}\\
&=&
\displaystyle
2\log(|Z|)\cdot
\int_\RR\widehat q(Y_t)(1-\widehat q(Y_t))\,dt
&\\
&\leq&
\displaystyle
\frac{\log(|Z|)}{\phi m}
\int_\RR|\delta_{G'}(Y_t)|\,dt
&\text{by \eqref{eq:quotient-expansion}}\\
&=&
\multicolumn{2}{l}{
\displaystyle
\frac{\log(|Z|)}{\phi m}
\int_\RR
\sum_{e=zz'\in E(G')}
\left|\mathbf 1[z\in Y_t]-\mathbf 1[z'\in Y_t]\right|\,dt
}\\
&=&
\multicolumn{2}{l}{
\displaystyle
\frac{\log(|Z|)}{\phi m}
\sum_{e=zz'\in E(G')}|f(z)-f(z')|
}\\
&\leq&
\displaystyle
\frac{\log(|Z|)}{\phi m}
\sum_{e=zz'\in E(G')}d'(z,z')
&\text{since }|f(z)-f(z')|\leq d'(z,z')\\
&\leq&
\displaystyle
\frac{\log(|Z|)}{\phi m}
\sum_{e=zz'\in E(G')}\ell'(e)
\\
&\leq&
\displaystyle
\frac{9\log(|Z|)}{\phi m}
&\text{by \eqref{eq:quotient-total-length}}\\
&\lesssim&
\displaystyle
\frac{\log(m/d^2)}{\phi m}
&\text{by \eqref{eq:center-count}.}
\end{array}
\]

By construction,
\[
d_\ell(u,v)
\leq
d_\ell(u,\kappa(u))
+
d'(\kappa(u),\kappa(v))
+
d_\ell(v,\kappa(v))
\]
for every $u,v\in V$.
Multiplying by $q(u)q(v)$ and summing over unordered pairs gives
\[
\begin{array}{rcll}
\displaystyle
\sum_{\{u,v\}\in\binom V2}q(u)q(v)d_\ell(u,v)
&\leq&
\multicolumn{2}{l}{
\displaystyle
\sum_{\{u,v\}\in\binom V2}q(u)q(v)
\left(
d_\ell(u,\kappa(u))
+
d'(\kappa(u),\kappa(v))
+
d_\ell(v,\kappa(v))
\right)
}\\
&=&
\multicolumn{2}{l}{
\displaystyle
\sum_{u\in V}
\sum_{\substack{v\in V\\v\neq u}}
q(u)q(v)d_\ell(u,\kappa(u))
+
\frac12
\sum_{u,v\in V}q(u)q(v)d'(\kappa(u),\kappa(v))
}\\
&=&
\multicolumn{2}{l}{
\displaystyle
\sum_{u\in V}q(u)d_\ell(u,\kappa(u))
\sum_{\substack{v\in V\\v\neq u}}q(v)
+
\frac12
\sum_{u,v\in V}q(u)q(v)d'(\kappa(u),\kappa(v))
}\\
&=&
\displaystyle
\sum_{u\in V}q(u)(1-q(u))d_\ell(u,\kappa(u))
+
\frac12
\sum_{z,z'\in Z}\widehat q(z)\widehat q(z')d'(z,z')
&\\
&\leq&
\displaystyle
\sum_{u\in V}q(u)d_\ell(u,\kappa(u))
+
\frac12
\sum_{z,z'\in Z}\widehat q(z)\widehat q(z')d'(z,z')
&\\
&\leq&
\displaystyle
\frac4m
+
\frac12
\sum_{z,z'\in Z}\widehat q(z)\widehat q(z')d'(z,z')
&\text{by \eqref{eq:center-distance}}\\
&\lesssim&
\displaystyle
\frac4m
+
\frac{\log(m/d^2)}{\phi m}
&\\
&\lesssim&
\displaystyle
\frac{\log(m/d^2)}{\phi m}.
\end{array}
\]
This proves \eqref{eq:expansion-average-distance}.

Choose $\alpha$ to be a sufficiently small constant multiple of
$\phi m/\log(m/d^2)$.
By \eqref{eq:expansion-average-distance},
for every $\ell\colon E\to\RR_{\geq0}$,
\[
\begin{aligned}
\sum_{\{u,v\}\in\binom V2}
\frac{\alpha\deg_G(u)\deg_G(v)}{4m^2}d_\ell(u,v)
&=
\alpha
\sum_{\{u,v\}\in\binom V2}q(u)q(v)d_\ell(u,v)\\
&\leq
\sum_{e\in E}\ell(e).
\end{aligned}
\]
This is the condition in \cref{lem:edge-flow-distance} for the demands in
\eqref{eq:transfer-product-flow} and unit edge capacities.
The lemma therefore gives a flow $F$ satisfying
\eqref{eq:transfer-product-flow} and $\ell_F(e)\leq1$ for every $e\in E$.
\end{proof}

\begin{lemma}
\label{lem:minor-demand-transfer}
Let $G=(V,E)$ be a $K_h$-minor-free graph with $h\geq3$,
and let $H$ be an $s$-vertex graph with $V(H)\subset V$ and $m\geq1$ edges.
Every unit $H$-flow $F$ in $G$ satisfies
\[
\con_2(F)^2
\gtrsim
\frac{m^3}{h^2s^2}.
\]
\end{lemma}

At a high level, the proof will build a kind of ``expander decomposition''
(delete edges so that the residual graph components are ``expanding'' and have a large minimum degree).
It will then apply the preceeding results to the residual components of this decomposition.

\begin{proof}
Fix $d:=m/(100s)$.
Let $\eta>0$ be a constant to be chosen later.
Delete vertices from $H$ that have degree less than $d$ until none remain.
If a remaining connected component $K$ with $m_K$ edges violates
\eqref{eq:transfer-expansion} for
\[
\phi_K
:=
\eta\frac{d\log(m_K/d^2)}{\sqrt{m_K}},
\]
delete the edges of a violating cut,
then repeat these two steps on the resulting components (using the same fixed $d,\eta$ values).

The vertex deletions remove fewer than $ds=m/100$ edges.
For a deleted cut,
let $S$ be its smaller-volume side.
We will require $\eta$ to be small enough that $\phi_K\leq\frac12$.
Then
\[
|\delta_K(S)|
<
\phi_K\operatorname{vol}_K(S)
\leq
\frac12\operatorname{vol}_K(S).
\]
Since
\[
\operatorname{vol}_K(S)
=
2|E(K[S])|+|\delta_K(S)|,
\]
we have $|E(K[S])|\geq\frac14\operatorname{vol}_K(S)$.
Charge the deleted cut edges equally to the edges of $K[S]$.
Each such edge receives charge
\[
\frac{|\delta_K(S)|}{|E(K[S])|}
\leq
\frac{|\delta_K(S)|}{\frac14\operatorname{vol}_K(S)}
<
4\phi_K.
\]
Whenever an edge is charged again,
its component has at most half as many edges as before.
Since low-degree vertices get deleted,
each such component has at least $d^2/2$ edges.
Fix an edge $e\in E(H)$.
For each $j\geq0$,
the edge $e$ is charged in at most one component $K$ with
\[
2^{j-1}d^2
\leq
|E(K)|
<
2^jd^2.
\]
The charge received in this component is at most
\[
\begin{aligned}
4\phi_K
&=
4\eta
\frac{d\log(|E(K)|/d^2)}{\sqrt{|E(K)|}}\\
&=
4\eta
\frac{\log(|E(K)|/d^2)}{\sqrt{|E(K)|/d^2}}\\
&\lesssim
\eta(j+1)2^{-j/2}.
\end{aligned}
\]
Thus the total charge to $e$ is
\[
\lesssim
\eta\sum_{j\geq0}(j+1)2^{-j/2}
\lesssim
\eta.
\]
Since each cut edge distributes one unit of charge,
the cuts remove $O(\eta m)$ edges.

Choose $\eta$ sufficiently small.
The remaining components $H_i$,
with $m_i:=|E(H_i)|$,
have minimum degree at least $d$,
satisfy \eqref{eq:transfer-expansion} with
$\phi_i:=\eta d\log(m_i/d^2)/\sqrt{m_i}$,
and
\[
\sum_i m_i
\gtrsim
m.
\]
Set $p_i(u):=\deg_{H_i}(u)/(2m_i)$ for $u\in V(H_i)$,
and extend $p_i$ by zero to $V$.
By \cref{prop:expansion-product-flow},
there is a flow $Q_i$ in $H_i$ with
$\ell_{Q_i}(e)\leq1$ for every $e\in E(H_i)$ and
\[
\begin{aligned}
Q_i[u,v]
&=
\alpha_i p_i(u)p_i(v),\\
\alpha_i
&\gtrsim
\frac{\phi_i m_i}{\log(m_i/d^2)}
=
\eta d\sqrt{m_i}
\gtrsim
d\sqrt{m_i}.
\end{aligned}
\]

For each $H_i$,
we will use both $F$ and $Q_i$ to construct a new multicommodity flow $R_i$.
For $e=\{u,v\}\in E(H_i)$,
let $F_e$ be the flow that agrees with $F$ on $\mathcal P_{uv}$
and is zero on every other path.
For every walk $W$ in $G$,
fix a path $\operatorname{red}(W)$ with the same endpoints
obtained by deleting cycles from $W$.
For each path $P'$ in $G$,
define
\[
R_i(P')
:=
\sum_{\substack{
P=u_0u_1\cdots u_t\\
P\text{ is a path in }H_i
}}
Q_i(P)
\left(
\sum_{\substack{
P_j\in\mathcal P_{u_{j-1}u_j}\text{ for }1\leq j\leq t\\
\operatorname{red}(P_1P_2\cdots P_t)=P'
}}
\prod_{j=1}^t F_{\{u_{j-1},u_j\}}(P_j)
\right).
\]
We will now show,
for every $\{u,v\}\in\binom{V(H_i)}2$,
that
\[
R_i[u,v]
=
\sum_{P'\in\mathcal P_{uv}}R_i(P')
=
Q_i[u,v].
\]
First,
fix a $u$--$v$ path $P=u_0u_1\cdots u_t$ in $H_i$.
Every choice of $P_1,\ldots,P_t$ occurs in exactly one inner sum,
according to its reduced path.
Therefore,
\[
\begin{aligned}
\sum_{P'\in\mathcal P_{u_0u_t}}
\left(
\sum_{\substack{
P_j\in\mathcal P_{u_{j-1}u_j}\text{ for }1\leq j\leq t\\
\operatorname{red}(P_1P_2\cdots P_t)=P'
}}
\prod_{j=1}^t F_{\{u_{j-1},u_j\}}(P_j)
\right)
&=
\sum_{\substack{
P_j\in\mathcal P_{u_{j-1}u_j}\text{ for }1\leq j\leq t
}}
\prod_{j=1}^t F_{\{u_{j-1},u_j\}}(P_j)\\
&=
\sum_{P_1\in\mathcal P_{u_0u_1}}
\cdots
\sum_{P_t\in\mathcal P_{u_{t-1}u_t}}
\prod_{j=1}^t F_{\{u_{j-1},u_j\}}(P_j)\\
&=
\prod_{j=1}^t
\left(
\sum_{P_j\in\mathcal P_{u_{j-1}u_j}}
F_{\{u_{j-1},u_j\}}(P_j)
\right)\\
&=
\prod_{j=1}^t F[u_{j-1},u_j]
=
1.
\end{aligned}
\]
The last equality holds because $H_i$ is a subgraph of $H$
and $F$ is a unit $H$-flow.
Multiplying by $Q_i(P)$ and summing over all $u$--$v$ paths $P$ in $H_i$
proves the claim.
In particular,
\[
R_i[u,v]
=
Q_i[u,v]
=
\alpha_i p_i(u)p_i(v).
\]
Fix $x\in V$.
We now substitute the definition of $R_i$ into the definition of congestion.
Deleting cycles cannot introduce $x$ into a path.
For a path $P=u_0u_1\cdots u_t$ in $H_i$ and $1\leq j\leq t$,
\[
\begin{aligned}
&\sum_{P_1\in\mathcal P_{u_0u_1}}
\cdots
\sum_{P_t\in\mathcal P_{u_{t-1}u_t}}
\left(
\prod_{k=1}^t F_{\{u_{k-1},u_k\}}(P_k)
\right)
\mathbf 1[x\in P_j]\\
&=
\left(
\sum_{P_j\in\mathcal P_{u_{j-1}u_j}}
F_{\{u_{j-1},u_j\}}(P_j)\mathbf 1[x\in P_j]
\right)
\prod_{\substack{1\leq k\leq t\\k\neq j}}
\left(
\sum_{P_k\in\mathcal P_{u_{k-1}u_k}}
F_{\{u_{k-1},u_k\}}(P_k)
\right)\\
&=
c_{F_{\{u_{j-1},u_j\}}}(x)
\prod_{\substack{1\leq k\leq t\\k\neq j}}
F[u_{k-1},u_k]
=
c_{F_{\{u_{j-1},u_j\}}}(x).
\end{aligned}
\]
This gives
\[
\begin{aligned}
c_{R_i}(x)
&=
\sum_{P'\in\mathcal P}R_i(P')\mathbf 1[x\in P']\\
&=
\sum_{\substack{
P=u_0u_1\cdots u_t\\
P\text{ is a path in }H_i
}}
Q_i(P)
\sum_{\substack{
P_j\in\mathcal P_{u_{j-1}u_j}\text{ for }1\leq j\leq t
}}
\left(\prod_{j=1}^t F_{\{u_{j-1},u_j\}}(P_j)\right)
\mathbf 1[x\in\operatorname{red}(P_1P_2\cdots P_t)]\\
&\leq
\sum_{\substack{
P=u_0u_1\cdots u_t\\
P\text{ is a path in }H_i
}}
Q_i(P)
\sum_{\substack{
P_j\in\mathcal P_{u_{j-1}u_j}\text{ for }1\leq j\leq t
}}
\left(\prod_{j=1}^t F_{\{u_{j-1},u_j\}}(P_j)\right)
\left(\sum_{j=1}^t\mathbf 1[x\in P_j]\right)\\
&=
\sum_{\substack{
P=u_0u_1\cdots u_t\\
P\text{ is a path in }H_i
}}
Q_i(P)\sum_{j=1}^t c_{F_{\{u_{j-1},u_j\}}}(x)\\
&=
\sum_{\substack{
P\text{ is a path in }H_i
}}
Q_i(P)\sum_{e\in E(P)}c_{F_e}(x)\\
&=
\sum_{e\in E(H_i)}c_{F_e}(x)
\sum_{\substack{
P\text{ is a path in }H_i\\
e\in E(P)
}}
Q_i(P)\\
&=
\sum_{e\in E(H_i)}\ell_{Q_i}(e)c_{F_e}(x).
\end{aligned}
\]
Since $\ell_{Q_i}(e)\leq1$ for every $e\in E(H_i)$
and the graphs $H_i$ are edge-disjoint,
\[
\sum_i c_{R_i}(x)
\leq
\sum_i\sum_{e\in E(H_i)}\ell_{Q_i}(e)c_{F_e}(x)
\leq
\sum_i\sum_{e\in E(H_i)}c_{F_e}(x)
\leq
\sum_{e\in E(H)}c_{F_e}(x)
=
c_F(x).
\]

Since $\sum_{u\in V(H_i)}\deg_{H_i}(u)=2m_i$
and $\deg_{H_i}(u)\leq m_i$,
$p_i$ is a probability distribution with $\|p_i\|_\infty\leq\frac12$.
Thus \cref{prop:minor-congestion} gives
\[
\con_2(R_i)
\gtrsim
\frac{\alpha_i}{h}.
\]

Finally,
for every $x\in V$,
the preceding bound and the nonnegativity of congestion give
\[
c_F(x)^2
\geq
\left(\sum_i c_{R_i}(x)\right)^2
\geq
\sum_i c_{R_i}(x)^2.
\]
Summing over $x\in V$ gives
\[
\begin{aligned}
\con_2(F)^2
&=
\sum_{x\in V}c_F(x)^2
\geq
\sum_i\sum_{x\in V}c_{R_i}(x)^2
=
\sum_i\con_2(R_i)^2\\
&\gtrsim
\frac1{h^2}\sum_i\alpha_i^2
\gtrsim
\frac{d^2}{h^2}\sum_i m_i
\gtrsim
\frac{m^3}{h^2s^2}.
\end{aligned}
\]
\end{proof}

\begin{proof}[Proof of \cref{thm:minor-subset-flow}]
Let $\mu$ and $F$ be feasible for the optimization problem defining
$\con_{2,r}(G)$.
For $0<t\leq1/r^2$,
let $H_t$ be the graph whose edges are the pairs $\{u,v\}$ with
$F[u,v]\geq t$ and whose vertices are the endpoints of these edges.
Write $m(t)=|E(H_t)|$ and $s(t)=|V(H_t)|$.

For every $u\in V(H_t)$,
there is some $v$ such that
\[
t
\leq
F[u,v]
=
\frac1{r^2}
\sum_{\substack{A\in\binom Vr\\u,v\in A}}\mu(A)
\leq
\frac1{r^2}
\sum_{\substack{A\in\binom Vr\\u\in A}}\mu(A).
\]
Summing over $u\in V(H_t)$ gives
\[
\begin{aligned}
s(t)\cdot t
&\leq
\frac1{r^2}
\sum_{u\in V(H_t)}
\sum_{\substack{A\in\binom Vr\\u\in A}}\mu(A)\\
&\leq
\frac1{r^2}
\sum_{u\in V}
\sum_{\substack{A\in\binom Vr\\u\in A}}\mu(A)\\
&=
\frac1{r^2}
\sum_{u\in V}
\sum_{A\in\binom Vr}\mathbf 1[u\in A]\mu(A)\\
&=
\frac1{r^2}
\sum_{A\in\binom Vr}\mu(A)
\sum_{u\in V}\mathbf 1[u\in A]\\
&=
\frac1{r^2}
\sum_{A\in\binom Vr}|A|\mu(A)\\
&=
\frac1r.
\end{aligned}
\]
Together with $s(t)\leq n$,
this gives
\[
s(t)
\leq
\min\left\{n,\frac1{rt}\right\}.
\]
Since $F[u,v]\leq1/r^2$ for every pair $\{u,v\}$,
\[
\begin{aligned}
\int_0^{1/r^2}m(t)\,dt
&=
\sum_{\{u,v\}\in\binom V2}
\int_0^{1/r^2}\mathbf 1[t\leq F[u,v]]\,dt\\
&=
\sum_{\{u,v\}\in\binom V2}F[u,v]\\
&=
\frac1{r^2}
\sum_{\{u,v\}\in\binom V2}
\sum_{\substack{A\in\binom Vr\\u,v\in A}}\mu(A)\\
&=
\frac1{r^2}
\sum_{A\in\binom Vr}\binom{|A|}{2}\mu(A)\\
&=
\frac{\binom r2}{r^2}.
\end{aligned}
\]
For every $t$ with $m(t)>0$,
\[
m(t)
=
\left(
\frac{t^2m(t)^3}{s(t)^2}
\right)^{1/3}
\frac{s(t)^{2/3}}{t^{2/3}}.
\]
It follows that
\[
\frac{\binom r2}{r^2}
\leq
\sup_{\substack{0<t\leq1/r^2\\m(t)>0}}
\left(
\frac{t^2m(t)^3}{s(t)^2}
\right)^{1/3}
\int_0^{1/r^2}
\frac{s(t)^{2/3}}{t^{2/3}}\,dt.
\]
Moreover,
the bounds on $s(t)$ give
\[
\begin{aligned}
\int_0^{1/r^2}
\frac{s(t)^{2/3}}{t^{2/3}}\,dt
&\leq
n^{2/3}
\int_0^{1/(nr)}t^{-2/3}\,dt
+
\frac1{r^{2/3}}
\int_{1/(nr)}^{1/r^2}t^{-4/3}\,dt\\
&\lesssim
\left(\frac nr\right)^{1/3}.
\end{aligned}
\]
Since $\binom r2/r^2\geq1/4$,
there is some $t$ such that
\[
\frac{t^2m(t)^3}{s(t)^2}
\gtrsim
\frac rn.
\]

For every edge $\{u,v\}$ of $H_t$,
scale the $u$--$v$ flow in $F$ down to have value exactly $t$,
and discard all other commodities.
Call the resulting flow $F_t$.
Then $F_t/t$ is a unit $H_t$-flow and
$c_{F_t}(x)\leq c_F(x)$ for every $x\in V$.
By \cref{lem:minor-demand-transfer},
\[
\con_2(F)^2
\geq
\con_2(F_t)^2
=
t^2\con_2(F_t/t)^2
\gtrsim
\frac{t^2m(t)^3}{h^2s(t)^2}
\gtrsim
\frac{r}{h^2n}.
\]
\end{proof}

\section{Eigenvalue Relationships}

In this section,
we prove that bounds on the reweighted eigenvalue $\lambda_k^*(G,\pi)$ are sufficient
to prove bounds on several other kinds of eigenvalues.

First, the following relationship between $\lambda_k^*(G)$,
$\lambda_k(L_G)$, and $\lambda_k(\mathcal L_G)$ is
well-understood (see e.g.~\cite{kwok2022cheeger,amini2018transfer}).

\begin{proposition}
\label{prop:standard-reweighted-relationships}
Let $G$ be a connected $n$-vertex graph with $n\geq2$ and maximum degree $\Delta$.
For every $1\leq k\leq n$,
\[
\lambda_k(\mathcal L_G)
\leq
\lambda_k(L_G)
\leq
\Delta\lambda_k^*(G).
\]
\end{proposition}

\begin{proof}
For the uniform probability distribution on $V(G)$,
the edge weighting $w_e=1/(\Delta n)$ is admissible.
Its $k$th eigenvalue is $\lambda_k(L_G)/\Delta$,
so $\lambda_k(L_G)\leq\Delta\lambda_k^*(G)$.

Let $D$ be the degree matrix of $G$.
By the Courant-Fischer theorem,
\[
\begin{aligned}
\lambda_k(\mathcal L_G)
&=
\min_{\substack{U\subset\RR^V\\\dim U=k}}
\max_{\substack{z\in U\\z\neq0}}
\frac{
  z^\top D^{-1/2}L_GD^{-1/2}z
}{
  z^\top z
}\\
&=
\min_{\substack{U\subset\RR^V\\\dim U=k}}
\max_{\substack{x\in U\\x\neq0}}
\frac{
  \sum_{uv\in E}(x_u-x_v)^2
}{
  \sum_{v\in V}\deg(v)x_v^2
}\\
&\leq
\min_{\substack{U\subset\RR^V\\\dim U=k}}
\max_{\substack{x\in U\\x\neq0}}
\frac{
  \sum_{uv\in E}(x_u-x_v)^2
}{
  \sum_{v\in V}x_v^2
}
=
\lambda_k(L_G).
\end{aligned}
\]
The second equality substitutes $x:=D^{-1/2}z$.
\end{proof}

The remainder of this section will be dedicated to showing that
our bounds on the reweighted eigenvalues $\lambda_k^*(G,\pi)$
imply bounds on Steklov eigenvalues $\sigma_k(G,B)$.

\begin{lemma}
\label{lem:steklov-schur-complement}
Let $G=(V,E)$ be connected,
and let $B\subset V$ be nonempty.
Let $\Omega=V\setminus B$.

Order the vertices as $B,\Omega$ and write
\[
L_G
=
\begin{pmatrix}
L_{BB} & L_{B\Omega}\\
L_{\Omega B} & L_{\Omega\Omega}
\end{pmatrix},
\qquad
\Lambda_{G,B}
:=
L_{BB}
-L_{B\Omega}L_{\Omega\Omega}^{-1}L_{\Omega B}.
\]
For $y\in\RR^B$,
set $Y_B=y$ and
$Y_\Omega=-L_{\Omega\Omega}^{-1}L_{\Omega B}y$.
Then $(L_GY)_\Omega=0$,
$\sigma_k(G,B)$ is the $k$th-smallest eigenvalue of $\Lambda_{G,B}$,
and
\[
\sum_{uv\in E}(Y_u-Y_v)^2
=
y^\top\Lambda_{G,B}y.
\]
\end{lemma}

\begin{proof}
If $z\in\ker L_{\Omega\Omega}$,
extend $z$ to $Z\in\RR^V$ by setting $Z_B=0$.
Then
\[
0
=
z^\top L_{\Omega\Omega}z
=
Z^\top L_GZ
=
\sum_{uv\in E}(Z_u-Z_v)^2.
\]
Thus, $Z$ is constant because $G$ is connected,
and $Z_B=0$ implies $z=0$.
Hence, $L_{\Omega\Omega}$ is invertible.

The definition of $Y_\Omega$ gives
\[
\begin{aligned}
(L_GY)_\Omega
&=
L_{\Omega B}y+L_{\Omega\Omega}Y_\Omega
=0,\\
(L_GY)_B
&=
L_{BB}y+L_{B\Omega}Y_\Omega
=
\Lambda_{G,B}y.
\end{aligned}
\]
For nonzero $Y$,
\cref{def:steklov-eigenvalues} requires
\[
(L_GY)_\Omega=0
\qquad\text{and}\qquad
(L_GY)_B=\sigma y.
\]
By our derivation, these are equivalent to requiring
$\Lambda_{G,B}y=\sigma y$, proving the claim.
Finally,
since $(L_GY)_\Omega=0$ and $Y_B=y$,
\[
\sum_{uv\in E}(Y_u-Y_v)^2
=
Y^\top L_GY
=
Y_B^\top(L_GY)_B+Y_\Omega^\top(L_GY)_\Omega
=
y^\top(L_GY)_B
=
y^\top\Lambda_{G,B}y.
\]
\end{proof}

\begin{theorem}
\label{thm:reweighted-to-steklov}
Fix $k\geq2$ and $A>0$.
Let $\mathcal G$ be a family of graphs closed under attaching leaves.
Suppose every $G\in\mathcal G$ with $|V(G)|\geq k$ satisfies
\[
\lambda_k^*(G)
\lesssim
\frac{A}{|V(G)|}.
\]
Then every connected $H\in\mathcal G$ and every $B\subset V(H)$ with $|B|\geq k$
satisfy
\[
\sigma_k(H,B)
\lesssim
\frac{A\Delta(H)}{|B|}.
\]
\end{theorem}

\begin{proof}
Let $H\in\mathcal G$, and denote $n=|V(H)|$, $\Delta=\Delta(H)$,
and
$\Omega=V(H)\setminus B$.

Choose $r=\Delta n^2$,
and attach $r$ leaves to every vertex of $B$,
call the resulting graph $G_r$, and let $N:=n+|B|r$ denote its vertex count.
Then $G_r\in\mathcal G$.
Give each original edge from $H$ weight $1/(2\Delta N)$
and each new edge weight $1/(2rN)$.
Let $w$ denote this edge weighting.
At any original vertex from $H$,
the incident weights sum to at most $1/(2N)+1/(2N)=1/N$,
while at any new vertex they are exactly $1/(2rN)<1/N$,
so $w$
is feasible for the outer optimization problem formulation of $\lambda_k^*(G_r)$
given in
\cref{obs:reweighted-weighted-laplacian}.

Let $L_w$ be the weighted Laplacian of $G_r$ with these edge weights,
and let $\mu_k$ be the $k$th eigenvalue of $NL_w$.
Fix $X\in\RR^{V(G_r)}$,
write $y=X_B$,
and let $Y\in\RR^{V(H)}$ be the vector specified for $L_H$, $B$, and $y$ in
\cref{lem:steklov-schur-complement}.
Set $Z=X_{V(H)}-Y$,
so $Z_B=0$.
Using both $(L_HY)_\Omega=0$ and $Z_B=0$,
\[
Z^\top L_HY
=
Z_B^\top(L_HY)_B
+
Z_\Omega^\top(L_HY)_\Omega
=0.
\]
By the choice of the edge weights,
\[
X^\top NL_wX
=
\frac{1}{2\Delta}
X_{V(H)}^\top L_HX_{V(H)}
+
\frac{1}{2r}
\sum_{\substack{xv\in E(G_r)\setminus E(H)\\v\in B}}
(X_x-y_v)^2.
\]
For the first term,
$Z^\top L_HY=0$ and
\cref{lem:steklov-schur-complement} give
\[
\begin{aligned}
X_{V(H)}^\top L_HX_{V(H)}
&=
(Y+Z)^\top L_H(Y+Z)
\\
&=
Y^\top L_HY
+
2Z^\top L_HY
+
Z^\top L_HZ
\\
&=
y^\top\Lambda_{H,B}y
+
Z^\top L_HZ
\\
&=
y^\top\Lambda_{H,B}y
+
\sum_{uv\in E(H)}(Z_u-Z_v)^2.
\end{aligned}
\]
Combining these identities gives
\[
X^\top NL_wX
=
\frac{1}{2\Delta}
\left(
y^\top\Lambda_{H,B}y
+
\sum_{uv\in E(H)}(Z_u-Z_v)^2
\right)
+
\frac{1}{2r}
\sum_{\substack{xv\in E(G_r)\setminus E(H)\\v\in B}}
(X_x-y_v)^2.
\]
For every $u\in\Omega$,
$(L_HY)_u=0$ means that $Y_u$ is the average of its neighbors' values.
If a maximum of $Y$ occurs in $\Omega$,
then every neighbor of a maximizing vertex has the same value,
so a path to $B$ shows that the maximum also occurs in $B$.
The same holds for the minimum.
Thus, $|Y_u|\leq\max_{v\in B}|y_v|$ for every $u\in V(H)$,
and
\[
\|Y\|^2
\leq
n\max_{v\in B}|y_v|^2
\leq
n\sum_{v\in B}y_v^2
=
n\|y\|^2.
\]
For each $u\in\Omega$,
choose a path $P_u$ from $u$ to $B$ with at most $n$ edges,
and orient it toward $B$.
Since $Z_B=0$,
telescoping gives $Z_u=\sum_{ab\in E(P_u)}(Z_a-Z_b)$.
Applying Cauchy-Schwarz,
we get
\[
\begin{aligned}
\|Z\|^2
&=
\sum_{u\in\Omega}Z_u^2\\
&\leq
\sum_{u\in\Omega}
|E(P_u)|
\sum_{ab\in E(P_u)}(Z_a-Z_b)^2\\
&\leq
n\sum_{u\in\Omega}
\sum_{ab\in E(P_u)}(Z_a-Z_b)^2\\
&\leq
n^2\sum_{uv\in E(H)}(Z_u-Z_v)^2.
\end{aligned}
\]
Using $(a+b)^2\leq2a^2+2b^2$,
the preceding bounds,
and $r\geq n$,
\[
\begin{aligned}
\|X\|^2
&=
\|X_{V(H)}\|^2
+
\sum_{\substack{xv\in E(G_r)\setminus E(H)\\v\in B}}X_x^2\\
&=
\|Y+Z\|^2
+
\sum_{\substack{xv\in E(G_r)\setminus E(H)\\v\in B}}
\left((X_x-y_v)+y_v\right)^2\\
&\leq
2\|Y\|^2
+2\|Z\|^2
+2\sum_{\substack{xv\in E(G_r)\setminus E(H)\\v\in B}}
(X_x-y_v)^2
+2r\sum_{v\in B}y_v^2\\
&\leq
2n\|y\|^2
+2n^2\sum_{uv\in E(H)}(Z_u-Z_v)^2
+2\sum_{\substack{xv\in E(G_r)\setminus E(H)\\v\in B}}
(X_x-y_v)^2
+2r\|y\|^2\\
&\leq
4r\|y\|^2
+2n^2\sum_{uv\in E(H)}(Z_u-Z_v)^2
+2\sum_{\substack{xv\in E(G_r)\setminus E(H)\\v\in B}}
(X_x-y_v)^2.
\end{aligned}
\]

This inequality holds for an arbitrary vector $X$.
Now restrict to vectors $X$ such that $y$ is orthogonal to the first $k-1$
eigenvectors of $\Lambda_{H,B}$.
This is a subspace of codimension $k-1$,
since there are $k-1$ linear constraints on $X$.
By \cref{lem:steklov-schur-complement},
$y^\top\Lambda_{H,B}y\geq\sigma_k(H,B)\|y\|^2$.
Moreover,
$\Lambda_{H,B}\preceq L_{BB}\preceq2\Delta I_B$,
so $\sigma_k(H,B)\leq2\Delta$,
and $r=\Delta n^2$.
Combining the preceding identity with the norm bound gives
\[
\begin{aligned}
\frac{\sigma_k(H,B)}{8\Delta r}\|X\|^2
&\leq
\begin{aligned}[t]
&
\frac{\sigma_k(H,B)}{2\Delta}\|y\|^2
+
\frac{\sigma_k(H,B)n^2}{4\Delta r}
\sum_{uv\in E(H)}(Z_u-Z_v)^2\\
&{}+
\frac{\sigma_k(H,B)}{4\Delta r}
\sum_{\substack{xv\in E(G_r)\setminus E(H)\\v\in B}}
(X_x-y_v)^2
\end{aligned}\\
&\leq
\begin{aligned}[t]
&
\frac{\sigma_k(H,B)}{2\Delta}\|y\|^2
+
\frac{1}{2\Delta}
\sum_{uv\in E(H)}(Z_u-Z_v)^2\\
&{}+
\frac{1}{2r}
\sum_{\substack{xv\in E(G_r)\setminus E(H)\\v\in B}}
(X_x-y_v)^2
\end{aligned}\\
&\leq
X^\top NL_wX.
\end{aligned}
\]
The Courant-Fischer theorem now gives
\[
r\mu_k
\geq
\frac{\sigma_k(H,B)}{8\Delta}.
\]
The hypothesis gives
\[
\mu_k
\leq
\lambda_k^*(G_r)
\lesssim
\frac{A}{N}.
\]
Combining these,
\[
\sigma_k(H,B)
\leq
8\Delta r\mu_k
\lesssim
\frac{A\Delta r}{N}
=
\frac{A\Delta r}{|V(H)|+r|B|}
\leq
\frac{A\Delta}{|B|}.
\]
\end{proof}

Note that
\cref{thm:other-eigenvalue-transfer}
follows from the combination of
\cref{prop:standard-reweighted-relationships},
\cref{thm:reweighted-to-steklov}
and
\cref{lem:graph-family-closure}.

\section{Acknowledgements}
LLMs were used extensively during the discovery process, although they played very little role in the final text of the paper itself beyond figures.
Most notably, a far more complex proof of a special case of \cref{thm:main-genus} was first suggested by GPT 5.5 ``extra high'' on June 27th, 2026,
using a generalization of the ``transfer principle''~\cite{amini2018transfer}
(as opposed to the method we present that lifts the existing bounds).
This marks the first paper with non-trivial LLM use outside of formalization for both the authors.
The authors independently checked all arguments and take responsibility for the results.

The techniques for the proof of
\cref{thm:main-minor},
on the other hand, are much more human and are based on intuition gained from an
understanding of the mechanisms underlying \cite{korhonen2024induced} and \cite{bonnet2026separator}.
The details were also worked out with aid from a large language model.

The general method for proving a result analogous to \cref{thm:reweighted-to-steklov}
was also human-devised, but the precise result we have presented (devised with aid from GPT 5.6 ``medium'') has slightly weaker premises than our original result, at the cost of a very slightly more tedious construction.

The Lean formalization was likewise developed with substantial automated assistance.
As one may guess from the dates mentioned here compared to the original
posting date of this work,
a simplification and acceptable presentation of these results took significant work for which large language models were not helpful.

\paragraph{Independent Work}

Independently of our work, Zhan and Zhou~\cite{zhan2026transfer} very recently posted a preprint
proving our \cref{cor:genus-other-eigen-bounds} directly for Steklov eigenvalues
(thereby also proving the other bounds in that statement).
They also claimed to have had some LLM-assistance in devising details of their result.
It appears that their methods are based on the more involved
transfer principle-based method mentioned just above.
We suspect that their result for Steklov eigenvalues
could be elevated to prove \cref{thm:main-genus}
by combining it with some kind of analogue of a converse for
\cref{thm:reweighted-to-steklov}.

\bibliographystyle{alpha}
\bibliography{ref}

\newpage
\appendix

\section{Lower Bounds}
\label{sec:lower-bounds}

Here we quickly review lower bounds on possible eigenvalue bounds for bounded genus and minor-free graphs,
even when the maximum degree is constant.
Consequently,
all of the bounded genus eigenvalue bounds we have given
(including \cref{thm:main-genus})
and all of the minor-free eigenvalue bounds we have given
(including \cref{thm:main-minor})
are optimal up to constant factors and $O(\text{polylog}(h))$ factors, respectively.

\begin{proposition}[{\cite[Remark~2.6]{amini2018transfer}}]
\label{prop:genus-lower-bounds}
For every sufficiently large $g$ and every $k\geq2$,
there are arbitrarily large $n$ and bounded-degree $n$-vertex graphs $G$ of orientable genus at most $g$
such that
\[
\lambda_k(\mathcal L_G)
\gtrsim
\frac{g+k}{n}.
\]
\end{proposition}

The following analogous lower bound for minor-free graphs is new (to the best of our knowledge),
but quite straightforward.

\begin{proposition}
\label{prop:minor-lower-bounds}
For every sufficiently large $h$ and every $k\geq2$,
there is a $K_h$-minor-free graph $G$ on
$n=\Theta(h^2k)$ vertices with maximum degree at most four
such that
\[
\lambda_k(L_G)\gtrsim1.
\]
\end{proposition}

Notably, the dependence on $k$ is inherently worse than the bound for genus-$g$ graphs.

\begin{proof}
Fix an even $r=\Theta(h^2)$ such that
$3r/2<\binom h2$.
Taking $d=3$ and $\varepsilon=1/10$ in
\cite[Theorem~1.3]{alon2021explicit},
there is a simple $3$-regular $r$-vertex graph $H$ such that
\[
\lambda_2(L_H)
\geq
3-2\sqrt2-1/10
\gtrsim1.
\]
The graph $H$ has $3r/2<\binom h2$ edges.
Since edge deletion and contraction cannot increase
the number of edges in a simple graph,
$H$ is $K_h$-minor-free.
Join $k-1$ disjoint copies of $H$ in a path using bridges,
using distinct vertices for the bridges incident to each copy,
as shown in \cref{fig:minor-lower-bound}.
Call the resulting graph $G$.
Then
$n=(k-1)r=\Theta(h^2k)$ and $\Delta(G)\leq4$.

\begin{figure}[h]
\centering
\begin{tikzpicture}[
  every node/.style={font=\small,text=figureink},
  vertex/.style={circle,fill=figureink,inner sep=1.8pt},
  copy/.style={draw=figureink!70,fill=figurepanel!45,rounded corners=6pt,
    minimum width=2.2cm,minimum height=1.7cm},
  edge/.style={draw=figureink,line width=0.7pt},
  bridge/.style={draw=routeorange,line width=1.4pt}
]
  \foreach \x/\name in {
    0/one,
    3/two,
    6/three,
    9/four} {
    \node[copy] (\name-box) at (\x,0) {};
    \node[vertex] (\name-left) at (\x-0.75,0) {};
    \node[vertex] (\name-top) at (\x-0.15,0.5) {};
    \node[vertex] (\name-right) at (\x+0.75,0) {};
    \node[vertex] (\name-bottom) at (\x+0.15,-0.5) {};
    \node[vertex] (\name-center) at (\x,0) {};
    \draw[edge]
      (\name-left)--(\name-top)--(\name-right)--(\name-bottom)--cycle;
    \draw[edge] (\name-top)--(\name-center)--(\name-bottom);
  }
  \draw[bridge] (one-right)--(two-left);
  \draw[bridge] (two-right)--(three-left);
  \draw[bridge] (three-right)--(four-left);
\end{tikzpicture}
\caption{The construction for $k=5$.}
\label{fig:minor-lower-bound}
\end{figure}
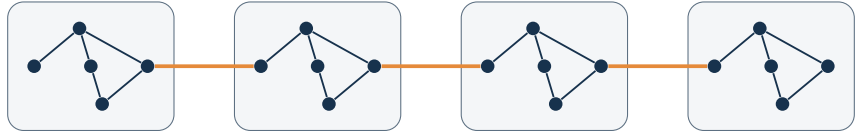

Every bridge of $G$ separates two collections of copies of $H$.
If we were to perform vertex deletions, edge deletions, and edge contractions on $G$,
either the two sides of each bridge would remain separated,
or the contracted bridge would be a cut vertex whenever vertices remain on both sides.
Thus,
every $2$-connected minor of $G$ with at least three vertices is a minor of a single copy of $H$.
Since $K_h$ is $2$-connected and $H$ is $K_h$-minor-free,
$G$ is $K_h$-minor-free.

Let $G'$ be the disjoint union of the copies of $H$.
Then $G'$ has exactly $k-1$ zero eigenvalues and
$\lambda_k(L_{G'})=\lambda_2(L_H)$.
Moreover,
for every $x\in\RR^{V(G)}$,
\[
x^\top L_Gx
=
x^\top L_{G'}x
+\sum_{uv\in E(G)\setminus E(G')}(x_u-x_v)^2
\geq
x^\top L_{G'}x.
\]
Therefore, by Courant-Fischer,
\[
\lambda_k(L_G)
\geq
\lambda_k(L_{G'})
=
\lambda_2(L_H)
\gtrsim1.
\]
\end{proof}

\end{document}